\documentclass{amsart}
\usepackage{amsthm,amssymb,amsmath,enumerate,graphicx, tikz}
\usepackage{amscd}
\usepackage{setspace}
\usepackage{comment}
\usepackage{hyperref}
\newtheorem{theorem}{Theorem}[section]

\newtheorem* {theorem*}{Theorem} 
\newtheorem{proposition}[theorem]{Proposition}

\newtheorem{lemma}[theorem]{Lemma}

\newtheorem{observation}[theorem]{Observation}

\newtheorem{corollary}[theorem]{Corollary}
\newtheorem{claim}[theorem]{Claim}  
\newtheorem{conjecture}[theorem]{Conjecture}
\theoremstyle{definition}
\newtheorem{definition}[theorem]{Definition}

\theoremstyle{remark}
\newtheorem{remark}[theorem]{Remark}
\newtheorem{example}[theorem]{Example}

\newcommand{\ctau}{\bar{\tau}}

\newcommand{\cnu}{\bar{\nu}}

\newcommand{\rvp}{\mathbb{R}^V_{\ge 0}}

\newcommand{\vh}{\vec{w}}
\newcommand{\vz}{\vec{z}}
\newcommand{\vw}{\vec{w}}

\newcommand{\vone}{\vec{1}}

\newcommand{\cp}{\mathcal{P}}
\newcommand{\D}{\mathcal{D}}
\newcommand{\G}{\mathcal{G}}
\newcommand{\etabar}{\bar{\eta}}

\newcommand{\K}{\mathcal{K}}

\newcommand{\ch}{\mathcal{H}}
\newcommand{\cm}{\mathcal{M}}

\newcommand{\vOne}{\vec{1}}

\newcommand{\C}{\mathcal{C}}

\newcommand{\EL}{\mathcal{L}}

\newcommand{\I}{\mathcal{I}}

\newcommand{\M}{\mathcal{M}}

\newcommand{\A}{\mathcal{A}}
\newcommand{\B}{\mathcal{B}}
\newcommand{\N}{\mathcal{N}}
\newcommand{\ck}{\mathcal{K}}

\newcommand{\R}{\mathbb{R}}

\newcommand{\T}{\mathcal{T}}

\newcommand{\vx}{\vec{x}}

\newcommand{\cT}{\Delta_\eta}
\newcommand{\E}{\Delta_r}

\newcommand{\refT}[1]{Theorem~\ref{#1}}
\newcommand{\refC}[1]{Corollary~\ref{#1}}

\newcommand{\refL}[1]{Lemma~\ref{#1}}
\newcommand{\refR}[1]{Remark~\ref{#1}}
\newcommand{\refS}[1]{Section~\ref{#1}}

\newcommand{\refD}[1]{Definition~\ref{#1}}
\newcommand{\refE}[1]{Example~\ref{#1}}

\newcommand{\refApp}[1]{Appendix~\ref{#1}}
\newcommand{\refCon}[1]{Conjecture~\ref{#1}}

\newcommand{\refOb}[1]{Observation~\ref{#1}}

\newcommand{\refProp}[1]{Proposition~\ref{#1}}

\usepackage{graphicx} 

\title{Coloring, list coloring, and fractional coloring in intersections of matroids}
\author{Ron Aharoni}
\address{Faculty of Mathematics, Technion, Haifa 32000, Israel}
\email{raharoni@gmail.com}
\author{Eli Berger}
\address{Department of Mathematics, University of Haifa, Haifa 31905,  Israel}
\email{berger.haifa@gmail.com}
\author{He Guo}
\address{Department of Mathematics and Mathematical Statistics, Ume\r{a} University, Ume\r{a} 90187, Sweden.}
\email{he.guo@umu.se}
\author{Dani Kotlar}
\address{Computer Science Department, Tel-Hai College, Upper Galilee 12210, Israel}
\email{dannykot@telhai.ac.il}
\date{July 2024; Revised in April, 2025}

\begin{document}

	\begin{abstract}
		It is known that in matroids the difference between the chromatic number and the fractional chromatic number is smaller than $1$, and that the list chromatic number is equal to the chromatic number. We investigate the gap within these pairs of parameters for hypergraphs that are the intersection of a given number $k$ of matroids. 
		We prove that in such hypergraphs the list chromatic number is at most $k$ times the chromatic number and at most  $2k-1$ times the maximum chromatic number among  the~$k$ matroids. We study the relationship between three polytopes associated with $k$-sets of matroids, and connect them to bounds on the fractional chromatic number of the intersection of the members of the $k$-set.  This also connects to bounds on the matroidal matching and covering number of the intersection of the members of the $k$-set.
		The tools used are in part topological.
	\end{abstract}
	\maketitle

	\section{Preliminaries}
	\subsection{Hypergraphs and complexes}

	A \emph{hypergraph}~$\ch$ is a collection of subsets, called {\em edges}, of its {\em vertex set} (or {\em ground set}), a finite set~$V=V(\ch)$. 
	Throughout the paper we assume that there is no isolated vertex in~$\ch$, i.e., every vertex $v\in V(\ch)$ belongs to some edge of~$\ch$. 
	If all edges of a hypergraph~$\ch$ are of the same size $k$ we say that $\ch$ is $k$-{\em uniform}, or that it is a {\em $k$-graph}.
	A hypergraph~$\ch$ is \emph{$k$-partite} if its vertex set can be divided into~$k$ parts~$V_1,\dots, V_k$ so that for every edge $S\in\ch$, $|S\cap V_i|=1$ for each $1\le i\le k$.
	For $U\subseteq V$, $\ch[U]=\{S\in \ch\mid S\subseteq U\}$ is the \emph{subhypergraph} of~$\ch$ \emph{induced} on $U$.

	A {\em matching} in a hypergraph is a set of disjoint edges, and a {\em cover} is a set of vertices meeting all its edges. The set of matchings in~$\ch$ is denoted by~$\M(\ch)$. The maximum size of a matching in~$\ch$ is denoted by~$\nu(\ch)$, and the minimum size of a cover by~$\tau(\ch)$. Obviously, $\tau(\ch) \ge \nu(\ch)$. 
	
	A set of vertices of $\ch$ is said to be {\em independent} if it does not contain an edge. The complex of independent sets in $\ch$ is denoted by $\I(\ch)$. 
	Clearly, both~$\M(\ch)$ and~$\I(\ch)$ are non-empty (both include the empty set) and closed under taking subsets. Hypergraphs satisfying these two conditions are called 
	{\em (abstract simplicial) complexes}, and their edges are also called {\em faces}. 
	This terminology is borrowed from topology. 
	
	We assume (except for one explicit deviation --- Definition \ref{join} of the ``join'' below) that all complexes have the same set, denoted by~$V$, as their ground set.  
	
	For a complex~$\C$ and $U\subseteq V$, let $rank_{\C}(U)=\max_{S\in \C[U]}|S|$. The \emph{rank} of~$\C$, denoted by~$rank(\C)$, is $rank_{\C}(V)$.
	
	A complex $\C$ is said to be a {\em flag complex} if it is $2$-determined, meaning that $e \in \C$ whenever $\binom{e}{2} \subseteq \C$. (Note that $\binom{S}{m}=\{T \subseteq S \mid |T|=m\}$.)

	\begin{definition}\label{join}
		The {\em join} $\C*\D$ of two complexes on disjoint ground sets is  $\{A \cup B \mid A \in \C, B \in \D\}$. If $V(\C) \cap V(\D)\neq \emptyset$, we define $\C*\D$ by first making a copy of~$\D$ on a ground set that is disjoint from that of~$\C$, and then taking the join of~$\C$ and the copy of~$\D$. 
	\end{definition}
	
	\subsection{Colorings, list colorings, and fractional colorings}
	
	Given a complex~$\C$, a \emph{coloring} by $\C$ is a set of faces of~$\C$ whose union is the ground set~$V$.
	The \emph{chromatic number} $\chi(\C)$ of~$\C$ is the minimum size (number of faces) of a coloring. 

	Let~$L_v$ be a set of {\em permissible colors} at every~$v \in V$. A {\em  list coloring} with respect to these is a function $f:V\rightarrow \cup_{v\in V}L_v$ satisfying $f(v)\in L_v$ for every~$v\in V$. It is said to be $\C$-{\em respecting} if $f^{-1}(c) \in \C$ for every color $c\in \cup_{v\in V}L_v$.
	The list chromatic number $\chi_\ell(\C)$ is the minimum number $p$ such that any system of lists $L_v$ of size $p$ has a $\C$-respecting list coloring. 
	
	Obviously, $\chi_\ell(\C) \ge \chi(\C)$. 
	
	The second main concept we shall study is that of fractional colorings.  
	
	\begin{definition}\label{def:fraccoloring}
		Given a complex~$\C$ on~$V$ and~$\vh\in\rvp$, a \emph{$\vh$-fractional coloring} of~$\C$ is a function~$f:\C\rightarrow\mathbb{R}_{\ge 0}$ satisfying $\sum_{S\in \C: v\in S}f(S)\ge w(v)$ for every $v \in V$.  The \emph{fractional $\vh$-chromatic number}, denoted by $\chi^*(\C,\vh)$, is the minimum of $\sum_{S\in\C}f(S)$ over all $\vh$-fractional colorings~$f$ of~$\C$.
		A $\vone$-fractional coloring is plainly called a \emph{fractional coloring}, and  
		$\chi^*(\C,\vone)$ is denoted by 
		$\chi^*(\C)$. 
	\end{definition}

	\subsection{Matroids}
	
	A complex $\M$ is called a {\em matroid} if  for all $S,T\in \M$ satisfying $|S|<|T|$ there exists $v\in T\setminus S$ such that $S\cup\{v\}\in \M$. 
	The edges of a matroid are said to be {\em independent}.  A \emph{base} of a matroid is a maximal independent set. A \emph{circuit}  
	is a minimal dependent set (the name comes from graphic matroids, namely matroids consisting of acyclic sets of edges in a graph). This is compatible with the terminology of independent sets in hypergraphs, once we note that $\M=\I(\ch)$ where $\ch$ is the set of circuits.
	
	For $A\subseteq V(\M)$, let
	\[span_{\M}(A)=A\cup\Big\{x\in V\mid \{x\}\cup S\not\in\M \text{ for some $S\in \M[A]$}\Big\}.\]
	$A\subseteq V$ is called {\em spanning} if $span_\M(A)=V$.

	Throughout the paper we assume that all matroids are loopless, namely all singletons are independent.

	Given a partition $\cp=(P_1, P_2, \ldots ,P_{m})$ of ~$V$,  let $\T(\cp)$ be the set of all subsets of~$V$ meeting each $P_i$ for $1\le i\le m$ in at most one vertex.~$\T(\cp)$ is called {\em partition matroid}.

	In a seminal paper \cite{ed1970},  Edmonds showed how combinatorial duality (min-max results)  can sometimes be formulated in terms of  the intersection of two matroids. The classical case is the K\"onig--Hall theorem, which can be viewed as a statement on the intersection of two partition matroids, defined on the edge set of a bipartite graph. The parts in each are the stars in one of the sides. A matching in this graph is a set of edges belonging to the intersection of the two matroids, and a marriage is a base in one matroid that is independent in the other.   
	
	Intersections of more than two  matroids  are more complex: $\min - \max$ results are no longer available, algorithms for finding maximum size sets in the intersection are no longer polynomial, and necessary and sufficient conditions for the existence of certain objects are replaced by sufficient conditions. 
	
	Let $\D^k$ be the collection of sets $\EL=\{\M_1, \ldots, \M_k\}$ of matroids. 
	Let $MINT_k=\{\bigcap \EL \mid \EL \in\D^k\}$, namely the set of complexes that are the intersection of $k$ matroids on the same ground set.

	The topic of the paper is properties of complexes belonging to $MINT_k$. Though possibly familiar, it is worthwhile mentioning the primary facts about such objects. The first, that was proved more than once, is that every complex is in some $MINT_k$.

	\begin{proposition}\label{prop:every}\cite{Kashiwabara, fekete, kortehausman}
		Every complex is  the intersection of matroids.
	\end{proposition}
	\begin{proof} 
		For $e\subseteq V$, let $\M_e= \{S \subseteq  V \mid S \not \supseteq e\}$. $\M_e$ is easily seen to be a matroid, and $\C=\bigcap_{e \not \in \C}\M_e$.
	\end{proof}
	
	In a beautiful M.Sc thesis~\cite{Imolaythesis}, Andr\'as Imolay addressed the question of how many matroids are needed in the intersection. Let~$\kappa(n)$ be the maximum, over all complexes~$\C$ on~$n$ vertices, of the minimum number of matroids whose intersection is~$\C$. 
	\begin{theorem}[Theorem 6.7 in~\cite{Imolaythesis}]
		$\binom{n-1}{\lfloor (n-1)/2\rfloor} \le \kappa(n) \le \binom{n}{\lfloor n/2\rfloor}$.
	\end{theorem}

	A special role will be played by systems of partition matroids. Given a $k$-partite hypergraph~$\ch$, we construct a collection $\EL=\{\M_1,\dots,\M_k \}$ of~$k$ partition matroids, each having ~$E(\ch)$  as ground set, where each parts of~$\M_i$ is a star $S_x$, namely the set of edges incident to a vertex $x$ in the~$i$th side of~$\ch$. This system of  matroids is denoted  by  $\EL(\ch)$. 
	
	Conversely,  let~$\EL = \{\M_1, \M_2, \ldots ,\M_k\}$ be a system of partition matroids, defined by the partitions $\cp^i=(P^i_1, P^i_2, \ldots ,P^i_{m_i})$ of the ground set~$V$ for $i\in [k]$.  Then the complex $\cap_{i=1}^k\M_i$ is the matching complex of a $k$-partite $k$-uniform hypergraph $\ck=\ck(\EL)$, whose vertices are the parts $P^i_j$, and each of its sides $S_i$ is $\{P^i_j\mid 1\le j \le m_i\}$. Every vertex $v\in V$ corresponds to an edge $e(v)=\{P^i_j \mid 1\le i\le k, v \in P^i_j\}$ of~$\ck$. Then $\cap_{i=1}^k \M_i = \M(\ck)$, the set of matchings in~$\ck$. 
	Clearly, 
	\begin{equation}\label{eq:KL}
		\EL(\K(\EL))=\EL\quad \text{ and }\quad \K(\EL(\K))=\K.
	\end{equation}

	These constructions are useful in characterizing intersections of partition matroids.

	\begin{proposition}\label{prop:four_equiv}
		The following conditions are equivalent:

		\begin{enumerate}[(i)]
			\item \label{eq:statement1} $\C=\I(G)$ for some graph $G$, 
			\item \label{eq:statement2} $\C$ is a flag complex,
			\item \label{eq:statement3} $\C$ is the intersection of $k$ partition matroids for some~$k$,
			\item  \label{eq:statement4} $\C$ is the matching complex of a $k$-partite hypergraph for some~$k$. 
		\end{enumerate}
		Moreover,~$\eqref{eq:statement3}$ and $\eqref{eq:statement4}$ are equivalent for each~$k$ separately. 
	\end{proposition}
	\begin{proof}
		$\eqref{eq:statement1} \Rightarrow \eqref{eq:statement2}$ is true by the definition of an independent set as a set in which each pair of elements is independent.
		To prove $\eqref{eq:statement2} \Rightarrow \eqref{eq:statement1}$ note that if~$\C$ is a flag complex then $\C=\I(G)$ for the graph whose edge set is $\{xy \mid xy \not \in \C\}$. $\eqref{eq:statement4} \Rightarrow \eqref{eq:statement1}$ is true since the matching complex~$M(H)$ of a hypergraph~$H$ is~$\I(L(H))$, where~$L(H)$ is the line graph of the hypergraph~$H$. 
		$\eqref{eq:statement3} \Leftrightarrow \eqref{eq:statement4}$ follows from~\eqref{eq:KL}. To prove $\eqref{eq:statement2} \Rightarrow \eqref{eq:statement3}$ note that if $|e|=2$ then the matroid~$\M_e$ in the proof of~\refProp{prop:every} is a partition matroid, with parts~$e$ and all singletons disjoint from~$e$.
	\end{proof}

	\subsection{Polytopes}
	Another viewpoint on the intersection of matroids, introdcued by Edmonds, is that  of polytopes, which are particularly useful in studying fractional colorings. 
	
	For a subset~$A$ of~$V$ let $\mathbf{1}_A \in \mathbb R^V$ be the \emph{characteristic function} of~$A$, namely the function taking value~$1$ on elements of~$A$ and~$0$ elsewhere.
	Functions will also be viewed as vectors, so a real-valued function $f$ on a set~$V$ is also denoted by $\vec{f}\in \mathbb{R}^V$. We write~$f[V]$ for $\sum_{v\in V}f(v)$.
	
	The polytope $P(\C)$ of a complex $\C$ on~$V$
	is the convex hull in~$\rvp$ of the characteristic vectors of the edges of $\C$. 
	
	A  polytope $Z\subseteq \rvp$ is \emph{closed down} if~$z\in Z$ and $0\le y\le z$ imply  $y\in Z$.
	\begin{observation}\label{ob:PCcloseddown}
		$P(\C)$ is closed down. 
	\end{observation}
	\begin{proof}
		Let     
		\begin{equation}\label{eq:closeddown}
			\vec{v} =\sum_{S\in\C} \lambda_S \mathbf{1}_{S} \in P(\C),
		\end{equation}
		where $\lambda_S\ge 0$ and $\sum_{S\in \C}\lambda_S=1$.
		Let $\vec{0}  \le \vec{u} \le  \vec{v}$. We claim that $\vec{u} \in P(\C)$.   Using induction, it suffices to consider the case that $\vec{u}$ and $\vec{v}$ differ just in one coordinate, namely $\vec{u}=\vec{v}-\alpha e_x$ for some $x \in V$. Furthermore, by the convexity of~$P(\C)$, it is enough to prove the case $\vec{u}(x)=0$. Since $\C$ is a complex, we can replace all $S$'s containing ~$x$ in~\eqref{eq:closeddown} by $S\setminus\{x\}$  to obtain $\vec{u}$.    
	\end{proof}

	\subsection{Matroids intersection vs. matroidal cooperative covers} 
	Given a $k$-set of matroids~$\EL=\{\M_1, \ldots ,\M_k\}\in \D^k$, let $\cnu(\EL)=rank(\bigcap \EL)$ and 
	$\ctau(\EL)=\min\{\sum_{i=1}^k rank_{\M_i}(V_i) \mid 
	\cup_{i=1}^k V_i=V\}$.
	A set belonging to $\bigcap \EL$ is called a \emph{matoridal matching} of~$\EL$, and a $k$-tuple of functions $(\mathbf{1}_{V_1}, \ldots,\mathbf{1}_{V_k})$ for which $\cup_{i=1}^k span_{\M_i}(V_i)=V$ is called a \emph{matroidal cooperative cover}. This is the reason for the notation~$\bar{\nu}$ and~$\bar{\tau}$ (see~\refS{sec:weightedmatroid} for more details).

	Given a maximum size set~$X$ in~$\bigcap \EL$ and any representation of~$V$ as $\cup_{i=1}^k V_i$ we have 
	\[ |X|=|X\cap(\cup_{i=1}^k)V_i|\le \sum_{i=1}^k|X\cap V_i|\le \sum_{i=1}^k rank_{\M_i}(V_i).  \]
	This shows that 
	\begin{equation}
		\cnu(\EL)\le   \ctau(\EL).
	\end{equation}

	Edmonds' famous two-matroids intersection theorem is that  for $k=2$ equality holds.
	\begin{theorem}[\cite{ed1970}]\label{thm:edmondstwomatroids}
		If $k=2$ then $\ctau(\EL)=\cnu(\EL)$.
	\end{theorem}
	We shall later prove that equality holds also for fractional versions of~$\cnu$ and~$\ctau$, for every $k$. It is not hard to prove that $\ctau(\EL) \le k\cnu(\EL)$ --- we shall return to variants of this inequality, and also to the basic conjecture in the field, of which Edmonds' theorem is a special case.  
	\begin{conjecture}[Conjecture 7.1 in~\cite{matcomp}]\label{conj:matroidalryser}
		$\ctau(\EL) \le (k-1)\cnu(\EL)$.
	\end{conjecture}

	\cite[Theorem 7.3]{matcomp} yields the conjecture for $k=3$.

	If $\EL\in \D^k$ is a collection of partition matroids and $H=\K(\EL)$, which is a $k$-partite hypergraph, then $\cnu(\EL)=\nu(H)$ and $\ctau(\EL)=\tau(H)$ (\refOb{ob:cnu=nupartition}). 
	Thus a conjecture of Ryser, stating that in $k$-partite hypergraphs~$\tau \le (k-1)\nu$, is a special case of~\refCon{conj:matroidalryser}.

	
	

	

	\section{Preview}
	The unifying theme of the paper is that belonging to $MINT_k$ entails quantifiably ``tame'' behavior of the complex. This is manifested in two ways. One, studied in Sections~\ref{sec:toptools}--\ref{sec:main2}, 
	is that the list chromatic number is not far from the ordinary chromatic number. This is proved via an observation, that a topological tool used to bound the chromatic number works just as well for the list chromatic number. This leads to two results. One (\refT{thm:chiellCkchiC}) is that if $\C \in MINT_k$
	then 
	\begin{equation*}
		\chi_\ell(\C) \le k\chi(\C).
	\end{equation*}
	This solves a conjecture posed by Kir{\'a}ly~\cite{kiralyk} and also by B\'{e}rczi,~Schwarcz, and~Yamaguchi~\cite{berczi}. The other result (\refC{cor:chiell2kchi}) is that if $\C=\cap_{i=1}^k\M_i \in MINT_k$, then 
	\begin{equation*}
		\chi_\ell (\C) \le (2k-1)\max_{1\le i\le k} \chi(\M_i).
	\end{equation*}
	The proof of the latter requires new topological tools.    
	If all matroids $\M_i$ are partition matroids then the $2k-1$ factor can be replaced by~$k$ (\refC{cor:listcolk=2}), i.e.,
	\begin{equation*}
		\chi_\ell (\C) \le k\max_{1\le i\le k} \chi(\M_i).
	\end{equation*}
	The latter is probably true (and also not sharp) for general matroids.

	\refS{sec:fractionalcoloring}  and on deal with another aspect of the ``tame-ness'' of intersections of matroids --- the behavior of fractional colorings. The motivation comes from two theorems of Edmonds on the intersection of two matroids. One is \refT{thm:edmondstwomatroids}.    The other, closely related, is that for any two matroids, $\M$ and $\N$, the following holds.
	\begin{theorem}[\cite{edmonds1979polytope}]\label{thm:edmonds2matroidintersection}
		$P(\M \cap \N)=P(\M) \cap P(\N)$. 
	\end{theorem}

	An obvious challenge is to extend this theorem to any number of matroids. 
	In \refS{sec:weightedmatroid} we define two notions: 
	matchings and cooperative covers, as well as their fractional and  weighted versions, for $k$-sets of matroids. We prove (\refT{thm:nu*w=tau*w}) that the fractional weighted cooperative covering number equals the  
	fractional weighted matching number for any $k$-set of matroids.

	Following the footsteps of a  result of F\"uredi~\cite{furedi1981maximum} and its weighted version     proved by  F\"uredi, Kahn and Seymour~\cite{furedikahnseymour}, we consider a matroidal fractional weighted Ryser-type conjecture (\refCon{conj:fksmatroid}).  We prove  its equivalence to the  following conjecture on polytopes.  
	
	\begin{conjecture}
		For $\C=\cap_{i=1}^k\M_i\in MINT_k$, 
		\[(k-1) P(\C) \supseteq \cap_{i=1}^k P(\M_i).\] 
	\end{conjecture}
	\refT{thm:edmonds2matroidintersection} is the  case  $k=2$.
	
	In ~\refS{sec:fractionalcoloring} we show an equivalence to  yet another conjecture.  
	\begin{conjecture}\label{conj:chi*kint}
		For $\C=\cap_{i=1}^k\M_i\in MINT_k$ and~$\vh\in\rvp$,
		\begin{equation*}
			\chi^*(\C,\vh) \le (k-1)\max_{1\le i\le k} \chi^*(\M_i,\vh).
		\end{equation*}
	\end{conjecture}
	Using known results, we prove these conjectures for $k \le 3$  and for partition matroids (see~\refS{subsec:partition}). For general~$k$ and general matroids, we prove these conjectures with~$k$ replacing~$k-1$ 
	(Corollaries~\ref{cor:kRPratio}--\ref{cor:kchi*}).
	
	\refS{sec:weightedfractional} is devoted to the study of the fractional $\vw$-chromatic number~$\chi(\C,\vw)$ for  general complexes~$\C$.
	
	The last section is devoted to a brief discussion of the  combination of the two main themes --- a fractional version of list coloring.

	\section{A topological tool}\label{sec:toptools}
	\subsection{$\C$-transversals}
	
	Hall's theorem is about the existence of injective choice functions. A more general  notion is that of $\C$-transversals. 
	Given a complex $\C$ and subsets~$S_i$ of~$V(\C)$ for $i \in I$, a $\C$-{\em transversal} is a choice  function $x_i \in S_i$ for~$i \in I$, whose image $\{x_i \mid i \in I\}$ is a face of~$\C$.
	
	This notion is particularly useful in two contexts. 
	
	\begin{enumerate}
		\item {\em Rainbow matchings}. A choice function whose domain is sets of hypergraph edges, and its image is a matching,  is called a \emph{rainbow matching}.  A special case is matchings in {\em bipartite hypergraphs}. A hypergraph  $\B$ is said to be bipartite if there is $A \subseteq V(\B)$ meeting each edge at precisely one vertex: each  $a \in A$ can be considered as a set $S_a$ of edges, those that complement $a$ to an edge of $\B$. A rainbow matching for this system is plainly a matching in~$\B$. 
		
		\item {\em Colorings}. A classic construction transforms any $k$-coloring of~$\C$ (namely a set of $k$ edges whose union is $V(\C)$),  to $\C$-transversals. Take the join $\D=*_{j\in [k]}\C$  of~$k$ copies 
		of~$\C$, and for every $v\in V$ let $S_v$ be the set of copies of $v$. Then a $\D$-transversal~$T$ is a cover of~$V$ by~$k$ faces of~$\C$ --- the $i$th face being the image of~$T$ in the $i$th copy of~$\C$. This is what we call ``$k$-coloring''.
	\end{enumerate}
	\begin{remark}
		The earliest reference we know to this construction is in Welsh's $1976$ book, ``Matroid theory''~\cite{welsh2010matroid}, where it is used to prove Edmonds' two matroids intersection theorem. But it may well be older than that.
	\end{remark}
	
	The study of $\C$-transversals requires a more general tool than  Hall's theorem. If~$\C$ is a matroid, then such a tool is given by a theorem of Rado~\cite{rado}. For general complexes, a topological tool has been developed. The basic theorem in this direction is that a $\C$-transversal exists if (but not only if) the union of any~$k$ sets~$S_i$ induces a complex of connectivity at least~$k$ (replacing ``being of size at least~$k$'' in Hall's theorem). We next define this notion. 
	
	\subsection{Connectivity}
	We denote by $S^k$  the $k$-dimensional sphere, and by $B^k$ is the $k$-dimensional ball. So, $S^k$ is the boundary of $B^{k+1}$.
	$B^0$ is a single point, and accordingly $S^{-1}=\emptyset$.

	A topological space~$X$ is \emph{homotopically $k$-connected} if 
	for every $-1\le i \le k$, 
	every continuous function from the $i$-dimensional sphere $S^i$ to $X$
	can be extended to a continuous function from the ball $B^{i+1}$ to $X$. If $X=\emptyset$, then it is \emph{homologically $(-1)$-connected}. If $X\neq \emptyset$, it is 
	\emph{homologically $k$-connected} if  $\tilde{H}_i(X) =0$ for all $0\le i\le k$, where $\tilde{H}_i(X)$ is the reduced $i$th homology group of~$X$. 
	
	Let $\eta(X)$ (resp. $\eta_H(X)$) be
	the maximum~$k$ for which $X$ is homotopically $k$-connected (resp. homologically  $k$-connected),  plus 2. 
	We shall generally not distinguish between the two notions, one reason being a result of Hurewicz.

	\begin{theorem}[Hurewicz, \cite{hatcher}]\label{etahbig}
		$\eta_H \ge \eta$. If $\eta(X)\ge 3$ (which signifies``$X$ is simply connected"), then $\eta(X)=\eta_H(X)$.
	\end{theorem}
	
	An abstract complex $\C$ has a geometric realization $||\C||$, obtained by positioning its vertices in general position in $\R^n$ for large enough $n$ ($n=2\cdot rank(\C)-1$ suffices, for example any rank-2 complex, namely a graph, can be realized without fortuitous intersections in $\mathbb R^3$) and representing every face by the convex hull of its vertices, the ``general position'' condition guaranteeing that there are no unwanted intersections. We write $\eta(\C)$ for $\eta(||\C||)$, and $\eta_H(\C)$ for $\eta_H(||\C||)$.

	Intuitively,~$\eta(X)$ is the minimum dimension of a hole in~$X$. For example, 
	$\eta(S^n)=n+1$ because the hole, namely the non-filled ball, is of dimension $n+1$. For the disk~$B^n$, $\eta(B^n)=\infty$, because there is no hole, meaning that images of $S^n$ can be filled for every $n$. $\eta(X)=0$ means $X=\emptyset$, and $\eta(X)=1$ means that~$X$ is non-empty and is not path-connected. 
	For another example, $\eta(\C)\ge 2$ means path-connectedness. The ``2'' is the size of the simplices that are used for filling those holes that are fillable. This is valid also for higher dimensions, which makes~$\eta$  a natural parameter in combinatorial settings.

	\subsection{Lower bounds on $\eta$.}
	To apply~\refT{thm:tophall} below, one needs combinatorially formulated lower bounds on~$\eta(\C)$. The following are the bounds used in this paper.
	
	\begin{proposition}\label{prop:etaHjoin}
		$\eta_H(\C*\D) =\eta_H(\C)+\eta_H(\D)$.
	\end{proposition}
	Homotopic~$\eta$ does not yield equality here, but an inequality. Luckily, the more useful of the two inequalities forming the equality. 
	\begin{proposition}\label{prop:etajoin}
		$\eta(\C*\D) \ge\eta(\C)+\eta(\D)$.
	\end{proposition}
	
	See \cite{matcomp} for a reference.

	The connection of~$\eta$ to matroids is given by a result of Whitney~\cite{whitney}.
	
	\begin{proposition}\label{prop:etamatroid}
		If $\M$ is a matroid then $\eta(\M)\ge rank(\M)$, with equality holding unless $\M$ has a co-loop, namely an element belonging to all bases, in which case $\eta(\M)=\infty$.
	\end{proposition} 
	The last statement is true since the matroid can be contractible to the co-loop.  
	A generalization  proved in~\cite{matcomp} will play  a central role in our investigation.
	
	\begin{proposition}\label{prop:etarankk}
		If  $\C \in MINT_k$, then $\eta(\C) \ge \frac{rank(\C)}{k}$.
	\end{proposition}
	\begin{example}[witnessing sharpness of the inequality]
		Let $A$ be a set of size $k$, and let $b$ be an element not belonging to $A$. Let  $\C =\cp(A)\cup\{\{b\}\}$. For every $a \in A$ let~$\M_a$ be the matroid whose bases are~$A$ and $(A \setminus\{a\})\cup \{b\}$.   
		Then $\C =\bigcap_{a \in A}\M_a$, so $\C \in MINT_k$.  We have $rank(\C)=k$ and $\eta(\C)=1$, since $\C$ is not connected. 
	\end{example}
	For a graph $G$ let $i\gamma(G)$ be the maximum over all sets $I \in \I(G)$ of the number of vertices needed to dominate $I$.
	The first combinatorial bound obtained on $\eta$  was implicit in \cite{ah}. 
	
	\begin{proposition}\label{prop:etaigamma}
		$\eta(\I(G))\ge i\gamma(G)$. 
	\end{proposition}

	The next result follows from~\refProp{prop:etarankk} and~\refProp{prop:four_equiv}.
	\begin{corollary}
		If  $\ch$ is $k$-uniform  then $\eta(\M(\ch))\ge \nu(\ch)/k$. 
	\end{corollary}

	In  \cite{aharonibergermeshulam} this was strengthened. 
	\begin{theorem}[Aharoni--Berger--Meshulam~\cite{aharonibergermeshulam}]\label{thm:abm}
		For a $k$-uniform hypergraph~$\ch$,
		\[  \eta(\M(\ch))\ge\frac{\nu^*(\ch)}{k}. \]
	\end{theorem}

	Here~$\nu^*(\ch)$ is the fractional matching number of ~$\ch$ (see~\refD{def:fractionalmatching}).
	
	A classic tool for obtaining lower bounds on $\eta$ is an exact sequence of complexes, known as the ``Mayer--Vietoris sequence''. One of the consequences of the exactness is reminiscent of unimodularity. We do not know whether  there is an intrinsic connection.

	\begin{lemma}\label{lemma:MV}
		For any pair $\A,\B$ of complexes,
		\[ \eta_H(\A) \ge \min( \eta_H(\A\cup\B), \eta_H(\A\cap\B) ).  \]
	\end{lemma}

	A homotopic version was proved in~\cite{matcomp}.

	Meshulam~\cite{meshulam1} showed how \refProp{prop:etaigamma} can be derived from this inequality. Another bound 
	\[\eta_H(\I(G)) \ge \gamma^E(G)\] 
	follows as a corollary, where~$\gamma^E(G)$ be the minimum size of a set of edges whose union dominates~$G$.
	Here, too, a homotopic version was proved in~\cite{ack}.  
	\begin{proposition}\label{gammae}
		$\eta(\I(G)) \ge \gamma^E(G)$.
	\end{proposition}
	In~\refS{subsec:Meshulamhyper} we prove an extension to independence complexes of hypergraphs (see~\refT{thm:hyperMeshulam}).
	

	\subsection{Topological Hall}

	\begin{theorem}\label{thm:tophall}
		Let $\C$ be a complex,  and let $V_i$ for $1\le i \le m$ be subsets of $V(\C)$. If $\eta(\C[\bigcup_{i \in I} V_i]) \ge |I|$ for every $I\subseteq [m] $ then the sets $V_i$ have a $\C$-transversal. 
		
	\end{theorem}
	
	This was proved implicitly in~\cite{ah} and formulated explicitly by the first author (see remark following Theorem 1.3  in \cite{meshulam1}).
	The special case when~$\C$ is a matroid is Rado's theorem~\cite{rado}, mentioned above, in which~$\eta$ is replaced by its (almost) equal parameter, the rank (see \refProp{prop:etamatroid}). 
	
	Meshulam~\cite{meshulam1} proved the homological version of this theorem, which by~\refT{etahbig} is stronger.

	
	

	\section{Expansion and colorability}\label{sec:expansion}
	
	We adopt the conventions that    $\lceil\frac{c}{\infty}\rceil =1$ and $\frac{c}{0}=\infty$  whenever $c>0$.   
	
	\begin{definition}
		The \emph{rank expansion number} $\E(\C)$ of a complex $\C$ is 
		\[\max_{\emptyset \neq S \subseteq V}\frac{|S|}{rank_\C(S)}.\]
	\end{definition}
	Since covering a set~$S$ by edges of size at most $rank_\C(S)$ requires at least $\frac{|S|}{rank_\C(S)}$ edges, we have 
	
	\begin{equation}\label{eq:boundonchiC}
		\chi(\C) \ge \E(\C).
	\end{equation}
	
	$\E$ has a topological counterpart. 
	\begin{definition}\label{def:ten}
		The {\em topological expansion number} $\cT(\C)$ of a complex $\C$ is 
		\[\max_{\emptyset \neq S \subseteq V}\frac{|S|}{\eta(\C[S])}.\]
	\end{definition}

	For a complex~$\C$, we define \[\bar{\eta}(\C)=\min( \eta(\C), rank(\C)).\]
	
	The advantage of  $\etabar$ over  $\eta$ is that it is always finite.
	\begin{definition}\label{def:en}
		The {\em expansion number} $\Delta(\C)$ of a complex $\C$ is 
		\[\max_{\emptyset \neq S \subseteq V}\frac{|S|}{\etabar(\C(S))}.\]
	\end{definition}
	
	Then 
	\begin{equation}\label{eq:relationexpansion}
		\E(\C)\le \Delta(\C) \quad \text{ and }\quad \cT(\C)\le     \Delta(\C).
	\end{equation}

	The following was  proved by Edmonds~\cite{edmonds1965minimum},
	extending a theorem of Nash-Williams~\cite{williams} on graph arboricity. 
	\begin{theorem}\label{thm:williams}
		In any matroid $\M$ we have $\chi(\M)=\lceil \E(\M)\rceil$. 
	\end{theorem}

	For general complexes  an inequality holds.
	
	\begin{theorem}[Corollary 8.6 in \cite{matcomp}]\label{chidelta}
		$\chi(\C) \le \lceil\Delta(\C)\rceil$.
	\end{theorem}

	\section{List coloring}\label{sec:main1}
	
	Recall that given a complex $\C$ and a list $L_v$ of permissible colors at each ~$v \in V$,  a $\C$-respecting list coloring is a choice (from the lists) function $f:V\rightarrow \cup_{v\in V}L_v$ such that $f^{-1}(c) \in \C$ for every color $c\in \cup_{v\in V}L_v$.
	


	Virtually the same proof as that of \refT{chidelta} yields the following stronger result.
	
	\begin{theorem}\label{thm:chiellM}
		For any  complex~$\C$,
		\[ \chi_\ell(\C)\le\lceil\cT(\C)\rceil. \]
	\end{theorem}
	\begin{proof}
		For any system of lists $(L_v)_{v\in V(\C)}$ satisfying $|L_v|=\lceil \cT(\C) \rceil$, we shall find a $\C$-respecting list coloring.
		Let $J=\bigcup_{v\in V}L_v$ be the set of all the colors. For every color $j \in J$, let $F_j=\{v\in V(\C) \mid j \in L_v\}$ be the set of all elements such that the color~$j$ is in their lists. Form $W=\bigcup_{j\in J}\{j\}\times F_j$, which is the set of~$|L_v|$ copies of each~$v\in V(\C)$.
		For each~$j \in J$ let~$\C_j$ be the complex  $\{\{j\}\times \sigma \mid \sigma \in \C[F_j]\}$, namely a copy of $\C [F_j]$.
		Let $\D=*_{j \in J}\C_j$, namely the join of all $\C_j$ for $j\in J$, which is a complex on~$W$.
		
		For $v\in V(\C)$, let $W_v=\{(j,v)\mid j\in L_v \}$ be the set of all the copies of~$v$ in~$W$. For $I \subseteq V$, we claim that 
		\begin{equation}\label{eq:tophallcondition}
			\eta(\D[\cup_{v\in I}W_v]) \ge |I|.
		\end{equation}

		To prove the claim, by \refProp{prop:etajoin} we have
		\[\eta(\D[\cup_{v\in I}W_v]) \ge \sum_{j \in J}\eta(\C[F_j \cap I]).\]

		By the definition of $\cT(\C)$, we have 
		\[\eta(\C[F_j \cap I])\ge \frac{|F_j \cap I|}{\lceil \cT(\C)\rceil}.\] 
		Every vertex~$v$ of~$I$ appears in $|L_v|=\lceil \cT(\C)\rceil$ many $F_j$, hence $\sum_{j \in J}|F_j \cap I|=\lceil \cT(\C)\rceil\cdot |I|$. Hence 
		\[\eta(\D[\cup_{v\in I}W_v])\ge \sum_{j \in J}\eta(\C[F_j \cap I])\ge 
		\frac{\sum_{j \in J}|F_j \cap I|}{\lceil \cT(\C)\rceil}=
		\frac{\lceil \cT(\C)\rceil\cdot |I|}{\lceil \cT(\C)\rceil}=|I|,\]
		completing the proof of~\eqref{eq:tophallcondition}.
		
		By \refT{thm:tophall}, there exists a choice function~$\phi:V\rightarrow \cup_{v\in V}W_v$
		such that $\phi(v)\in W_v$ for each~$v\in V$ and $Im(\phi)\in \D$. Especially,~$Im(\phi)$ is of the form $\bigsqcup_{j\in J}\{j\}\times \sigma_j$ for $\sigma_j\in \C[F_j]\subseteq \C$. Coloring every vertex $v\in V$ by the color $j=\phi(v)$ (so that $v\in\sigma_j$) produces then the desired $\C$-respecting list coloring.
	\end{proof}
	
	Combining with~\refProp{prop:etamatroid} and \refT{thm:williams}, this yields the following result.
	\begin{theorem}
		For a matroid~$\M$, $\chi_\ell(\M) =\chi(\M)$. 
	\end{theorem}
	
	As noted above, this was first proved by Seymour~\cite{seymourmatroid}.

	In~\cite{berczi,kiralyk} it was conjectured  that there exists a constant~$\alpha$
	such that $\chi_\ell(\C)\le \alpha \chi(\C)$ for every $\C\in MINT_2$. Indeed, this is the case.
	\begin{theorem}\label{thm:chiellCkchiC}
		If $\C \in MINT_k$ then $\chi_\ell(\C) \le k\chi(\C)$.
	\end{theorem}
	\begin{proof}
		Assume that $S\subseteq V(\C)$ attains the maximum  in the definition of $\cT(\C)$, namely $\frac{|S|}{\eta(\C[S])}=\cT(\C)$.
		By~\refProp{prop:etarankk}, $rank_\C(S)\le k\eta(\C[S])$.
		Together with~\eqref{eq:boundonchiC}, this yields
		\[ k\chi(\C)\ge k\chi(\C[S]) \ge \Big\lceil k\frac{|S|}{rank_{\C}(S)}\Big\rceil \ge  \Big\lceil \frac{|S|}{\eta(\C[S])}\Big\rceil =\lceil \cT(\C) \rceil \ge \chi_\ell(\C),   \]
		where the last inequality is given by~\refT{thm:chiellM}.
		It completes the proof.
	\end{proof}

	\begin{conjecture}\cite{berczi,kiraly2013twomatroids}
		If $\C \in MINT_2$ then $\chi_\ell(\C) = \chi(\C)$.
	\end{conjecture}
	
	Galvin~\cite{GALVIN1995153} proved the conjecture for $\C=\M\cap \N$ where ~$\M$ and~$\N$ are partition matroids. More cases were studied in~\cite{berczi,kiraly2013twomatroids,hesolo}.

	\section{Bounding $\Delta_\eta(\bigcap\EL)$  for  partition matroids}\label{sec:main2partition}
	
	In this section, we prove an upper bound on $\Delta_\eta$,  thereby also on $\chi_\ell$ (by~\refT{thm:chiellM}), of the intersection of $k$ partition matroids. We   conjecture that the result holds also for the intersection of general matroids, but in the general case we can only prove a weaker result --- this will be done in the next section.

	\begin{theorem}\label{thm:Deltavsd}
		If $\EL=\{M_1, \ldots,\M_k\}\in\D^k$ is a set of $k$ partition matroids, then 
		\begin{equation}\label{eq:deltvsdelta}
			\cT(\bigcap \EL)\le k \max_{1\le i\le k}\E(\M_i).
		\end{equation}   
	\end{theorem}
	
	\begin{proof}
		Let $\ch=\ck(\EL)$, the $k$-partite hypergraph associated with $\EL$ (see the definition in the proof of~\refProp{prop:four_equiv}). Let $d=\max_{1\le i\le k}\E(\M_i)$. Then the maximum degree in~$\ch$ is~$d$,  hence putting weight $\frac{1}{d}$ on every edge of $\ch$ yields a fractional matching, so $\nu^*(\ch)\ge \frac{|E(\ch)|}{d}$. By 
		\refT{thm:abm} this implies  $\eta(\M(\ch))\ge \frac{|E(\ch)|}{kd}$, and thus $\frac{|E(\ch)|}{\eta(\M(\ch))} \le kd$. By the same token, this is true for any subhypergraph of~$\ch$, yielding~\eqref{eq:deltvsdelta}. 
	\end{proof}
	The inequality in \eqref{eq:deltvsdelta} is sharp, as shown by the following example.

	\begin{example}\label{example:P_k}
		The $k$-uniform affine plane is obtained by removing a line $L$ from the projective plane of uniformity $k+1$ (if such exists) and removing all vertices of ~$L$ from the other lines. 
		Thus the affine plane is a $k$-uniform hypergraph on $k^2$ vertices that is the union of $k+1$ matchings $M_1, \ldots ,M_{k+1}$, each of size $k$, that are pairwise cross-intersecting, namely $e \cap f \neq \emptyset$ whenever $e,f$ belong to different $M_i$s. For example, the affine plane for $k=2$ is $K_4$. 
		
		Choose one of the $k+1$ matchings, say $M_1$, and remove its edges (keeping the vertices). What remains is a $k$-partite hypergaph, which we name $Q_k$, whose sides are the edges of $M_1$. It is $k$-regular, namely each vertex is contained in~$k$ edges. Also, $|E(Q_k)|=k^2$. For example, $Q_2=C_4$.

		Let~$\EL=\EL(Q_k)$.
		By the pairwise cross-intersecting property, we have $\eta(\bigcap\EL)=\eta(\M(Q_k))=1$ and thus $\cT(\bigcap\EL)=|Q_k|=k^2$. On the other hand, the $k$-regularity of~$Q_k$ means~$\E(\M_i)=k$ so that  $\cT(\bigcap\EL)=k^2=k\max_{1\le i\le k}\E(\M_i)$.
	\end{example}
	Combining~\refT{thm:chiellM}, \refT{thm:Deltavsd} 
	and~\refT{thm:williams} yields the following corollary.
	\begin{corollary}\label{cor:listcolk=2}
		If $\C$ is the intersection of~$k$ partition matroids~$\M_1,\dots, \M_k$ on the same ground set, then $\chi_\ell(\C) \le k \max_{1\le i\le k}\chi(\M_i)$.
	\end{corollary}

	\section{Bounding $\cT(\bigcap \EL)$  for general matroids}\label{sec:main2}
	
	\subsection{A bound}
	As mentioned above, we suspect that \refT{thm:Deltavsd} is also valid  for $k$-tuples of general matroids, but  we can only prove a weaker result.

	\begin{theorem}\label{thm:Delta2k}
		Let $\EL=\{\M_1,\dots, \M_k\} \in \D^k$. Then 
		\begin{equation}\label{deltvsdeltagen}
			\cT(\bigcap\EL)\le (2k-1) \max_{1\le i\le k}\E(\M_i).
		\end{equation}
		
	\end{theorem}

	Combining with~\refT{thm:chiellM} 
	and~\refT{thm:williams} this yields
	\begin{corollary}\label{cor:chiell2kchi}
		Let $\EL=\{\M_1,\dots, \M_k\} \in \D^k$. Then 
		\[\chi_\ell(\bigcap\EL)\le (2k-1) \max_{1\le i\le k}\chi(\M_i).\]
	\end{corollary}

	The proof of~\cite[Theorem 8.9]{matcomp} shows that for $\{\M_1, \M_2\} \in \D^2$, $\cT(\M_1\cap\M_2)\le 2\max\big(\Delta(\M_1),\Delta(\M_2)\big)$, which together with~\refProp{prop:etamatroid} and~\refT{thm:chiellM} yields the following result.
	
	\begin{theorem}
		For  $\{\M_1, \M_2\} \in \D^2$, $\chi_\ell(\M_1 
		\cap \M_2) \le 2\max_{1\le i\le 2}\chi(\M_i)$.
	\end{theorem}
	
	In \cite{sumofchi} this  was strengthened to: 
	\[  \chi_\ell(\M_1\cap\M_2) \le \chi(\M_1)+ \chi(\M_2).   \]
	
	\begin{conjecture}\label{conj:chik-1chi}
		For  $\EL =\{\M_1, \ldots ,\M_k\} \in \D^k$
		there holds 
		\[\chi(\bigcap \EL) \le k \max_{1\le i \le k}\chi(\M_i).\]  
	\end{conjecture}

	In~\cite{matcomp}, it was conjectured that:
	$\chi(\M_1 
	\cap \M_2) \le \max(\chi(\M_1),\chi(\M_2)+1)$, and it is tempting to extend the conjecture to:
	
	\begin{conjecture}\label{conj:generalized_A_B_conj}
		For  $\EL =\{\M_1, \ldots ,\M_k\} \in \D^k$
		there holds 
		\[ \chi(\bigcap \EL) \le (k-1) \max_{1\le i \le k}\chi(\M_i) +1. \] 
	\end{conjecture}

	We now go back to the proof of \refT{thm:Delta2k}.

\subsection{A hypergraph Mayer-Vietoris inequality}\label{subsec:Meshulamhyper}

At the core of the proof of  ~\refT{thm:Delta2k} stands a Mayer-Vietoris type inequality.
To state it, we need the following operations on hypergraphs.
For a hypergraph~$\ch$ and  edge~$e \in\ch$, we write   $\ch-e$ for $\ch\setminus\{e\}$. 
For $X\subseteq V$,  the \emph{contraction} hypergraph $\ch/X$ has  vertex set ~$V\setminus X$ and  edge set ~$\{f\setminus X: f\in E, f\not\subseteq X  \}$. 

\begin{theorem}\label{thm:hyperMV}
	Let $\ch$ be a hypergraph and let $e$ be a containment-wise minimal edge of~$\ch$. Then
	\begin{equation}\label{genmeshulam} \eta(\I(\ch)) \ge\min\Big( \eta(\I(\ch-e)),\; \eta(\I(\ch/e))+|e|-1   \Big).  
	\end{equation}
\end{theorem}

\begin{proof}
	In ~\refL{lemma:MV}  set $\A=\I(\ch)$ and $\B=2^e*\I(\ch/e)$.
\end{proof}

To state~\refT{thm:hyperMeshulam}, we need the following definitions.
Given a hypergraph~$\ch$, a sequence $\ck=(e_1,\ldots ,e_p)$ of edges of~$\ch$ is {\em dominating} if 
for every vertex $v \in V\setminus \bigcup \ck$ there exists an edge~$f$ of~$\ch$ such that $f \setminus \bigcup \ck = \{v\}$. It is called {\em frugal } if  $|e_i \setminus \bigcup_{\ell<i}e_\ell|>1$ for every $1\le i \le p$. 
Let~$\gamma^E(\ch)$ be the minimum of $|\bigcup \ck|-|\ck|$ over all $\ck$  that are both frugal and dominating
($\infty$, if there is no such~$\ck$). 
\begin{theorem}\label{thm:hyperMeshulam}
	$\eta(\I(\ch)) \ge \gamma^E(\ch)$  for any hypergraph~$\ch$.
\end{theorem}

\begin{remark}
	This is a generalization of Proposition \ref{gammae}.
\end{remark}
\begin{proof}[Proof of~\refT{thm:hyperMeshulam}]
	Observe that if a hypergraph $\G$ has an edge of size 1, say~$\{v\}$, then
	\begin{equation}\label{eq:nosingle}
		\I(\mathcal{G})=\I\Big(\mathcal{G}[V(\mathcal{G})\setminus\{v\}]\Big).
	\end{equation}

	To prove the theorem, we apply the following algorithm. Let $\ch_0=\ch$. 
	Suppose that the hypergraph~$\ch_j$ is defined for some $j\ge 0$. Delete all singleton edges and the vertices forming them.  By~\eqref{eq:nosingle} this does not change the independence complex. All edges of the resulting 
	hypergraph ~$\ch_j^*$ have sizes at least 2.
	If ~$E(\ch_j^*)=\emptyset$  while $V(\ch_j^*)\neq \emptyset$, we stop. In this case $\I(\ch_j^*)=2^{V(\ch_j^*)}$ so that $||\I(\ch_j^*) ||\cong B^{|V(\ch^*)|-1}$ and
	\begin{equation}\label{eq:casecase1}
		\eta(\I(\ch_j^*))=\infty; 
	\end{equation}
	Or we stop if~$V(\ch_j^*)=\emptyset$, in which case
	\begin{equation}\label{eq:casecase2}
		\eta(\I(\ch_j^*))=0. 
	\end{equation}
	If neither of these  cases occurs, we proceed. Choose a containment-wise minimal edge~$f_j$ of~$\ch_j^*$.
	Then $|f_j|\ge 2$ and by~\refT{thm:hyperMV}
	
	\begin{equation*}
		\eta(\I(\ch_j^*)) \ge\min\Big( \eta(\I(\ch_j^*-f_j)),\; \eta(\I(\ch_j^*/f_j))+|f_j|-1   \Big).   
	\end{equation*}
	If $\eta(\I(\ch_j^*)) \ge  \eta(\I(\ch_j^*-f_j))$, then we set $\ch_{j+1}=\ch_j^*-f_j$; otherwise we set $\ch_{j+1}=\ch_j^*/f_j$ and in this case, the edge~$f_j$ is marked as ``contracted". And we repeat this process.
	
	Clearly we will finally stop either because the edge set becomes empty and the vertex set is not empty, or the vertex set becomes empty.
	If the former happens, by~\eqref{eq:casecase1}, inductively we have
	\[  \eta(\I(\ch))=\eta(\I(\ch_0))\ge \infty\ge \gamma^E(\ch) \]
	so that the theorem is true.
	If the latter happens, by~\eqref{eq:casecase2}, let~$e_i^*$ be those~$f_i$ that are contracted during the algorithm and assume they are $(e_1^*,\dots, e_p^*)$. Then by the construction,
	\begin{equation}\label{eq:sumei}
		\eta(\I(\ch) ) =\eta(\I(\ch_0)) \ge   \sum_{i=1}^p (|e_i^*|-1).
	\end{equation}
	Furthermore, since the deletions of vertices or edges does not change the remaining edges, while contracting~$f_i$ is deleting~$f_i$ from other edges, then we know that there exists edges $e_1,\dots, e_p$ of~$\ch_0=\ch$ such that for each $1\le i\le p$
	\[  e_{i}\setminus \cup_{\ell<i}e_\ell =e_i^*.  \]
	Let~$\ck=(e_1,\dots, e_p)$. Since~$|e_i^*|\ge 2$, we have~$\ck$ is frugal. And for every $v\in V(\ch)\setminus \bigcup\ck$,~$v$ is a deleted vertex during the algorithm, which means during the algorithm there exists an edge of size one consisting of~$v$. Thus for the same reason as above, there exists an edge~$f$ of~$\ch$ such that $f\setminus\bigcup \ck=\{v\}$. Therefore~$\ck$ is dominating.
	And we have 
	\[ |\bigcup \ck|-|\ck| = \sum_{i=1}^p|e_i\setminus \cup_{\ell<i}e_\ell| -p=\sum_{i=1}^p(|e_i^*|-1).    \]
	Therefore by~\eqref{eq:sumei}, $ \eta(\I(\ch))\ge\gamma^E(\ch),$ as desired.
\end{proof}

\subsection{Proof of~\refT{thm:Delta2k}}\label{sec:proofof2k-1}

\begin{observation}[e.g., Lemma 1.4.8 in \cite{oxley}]
	\label{ob:spancircuit}
	For any circuit $C$ in a matroid~$\M$, any ~$v\in C$, and any subset~$T$ of the ground set of~$\M$, 
	\[span_\M(T\cup C)=span_{\M}(T\cup C\setminus\{v\}).  \]
\end{observation}
The \emph{circuit hypergraph} $CIRC(\M)$ of a matroid $\M$ is the set of circuits in $\M$.  
By the definition of ``circuit'', if~$\ch=CIRC(\M)$ then $\M=\I(\ch).$

\begin{proof}[Proof of~\refT{thm:Delta2k}]
Let $t=\max_{1\le i\le k}\E(\M_i)$, $\C=\cap_{i=1}^k\M_i$, and $V$ be the common ground set of the matroids.
First we prove that
\[  \frac{|V|}{\eta(\C)}\le (2k-1)t. \]
Let~$\ch_i=CIRC(\M_i)$ for $1\le i\le k$ and
$\ch=\cup_{i=1}^k\ch_i$. Then
\[ \C=\I(\ch). \]

We may assume that $\gamma^E(\ch)$ is finite, otherwise by~\refT{thm:hyperMeshulam},
$\eta(\C)=\infty$ and we are done.   
Let $\ck=(e_1,\dots,e_p)$ be a dominating and frugal sequence of edges of~$\ch$ for which
\begin{equation} \label{eq:gammeEeq}
	\gamma^E(\ch)= |\bigcup\ck|-|\ck|=\sum_{j=1}^p|e_j\setminus\cup_{\ell<j}e_\ell|-p=\sum_{j=1}^p(|e_j\setminus\cup_{\ell<j}e_\ell|-1). 
\end{equation}
Since~$\ck$ is dominating, for any $v\in V\setminus\bigcup\ck$ there exists an edge~$f\in\ch$  such that
\[f\setminus\bigcup\ck=\{v\}. \]
Let~$i$ be such that $f\in\ch_i= CIRC(\M_i)$. Then $v\in span_{\M_i}(f\setminus\{v\})$, and then
\[ v\in \bigcup\ck\cup f\subseteq span_{\M_i}(\bigcup\ck \cup f\setminus\{v\})=span_{\M_i}(\bigcup\ck).  \]
Therefore 
\begin{equation}\label{eq:Vcontain}
	V\subseteq\cup_{i=1}^k span_{\M_i}(\bigcup\ck).  
\end{equation}
For each~$1\le j\le p$, noting that~$e_j$ is an edge of~$\ch=\cup_{i=1}^k\ch_i$, let
\[ L_j=\{i\mid \text{$e_j$ is an edge of~$\ch_i$}\}. \]
For each~$1\le i\le k$ and each index~$j$ such that $i\in L_j$, we choose an element~$v_{i,j}\in e_j\setminus\cup_{\ell<j}e_\ell$. Recall that $i\in L_j$ means that~$e_j$ is a circuit in~$\M_i$. Hence, by ~\refOb{ob:spancircuit}, we can delete $(v_{i,j})_{\{j:i\in L_j\}}$ one by one, in descending order of $j$, to obtain: 
\[ span_{\M_i}(\cup_{j:i\in L_j}e_j\cup\cup_{j:i\not\in L_j}e_j)=span_{\M_i}\Big(\cup_{j:i\in L_j}(e_j\setminus \{v_{i,j}\})\cup\cup_{j:i\not\in L_j}e_j\Big) \]
and then
\begin{align*}
	&span_{\M_i}(\bigcup\ck)\\
	=&span_{\M_i}(\cup_{j:i\in L_j}e_j\cup\cup_{j:i\not\in L_j}e_j)\\
	=&span_{\M_i}\Big(\cup_{j:i\in L_j}(e_j\setminus \{v_{i,j}\})\cup\cup_{j:i\not\in L_j}e_j\Big)\\
	=&span_{\M_i}\Big(\cup_{j:i\in L_j}\big(e_j\setminus (\{v_{i,j}\}\cup \cup_{\ell<j}e_\ell)\big)\cup\cup_{j:i\not\in L_j}(e_j\setminus \cup_{\ell<j}e_\ell)\Big), 
\end{align*}
which has size at most
\begin{equation}\label{eq:sizeofspanMi}
	t\Big(\sum_{j:i\in L_j}(|e_j\setminus \cup_{\ell<j}e_\ell|-1)+\sum_{j:i\not\in L_j}|e_j\setminus\cup_{\ell<j}e_\ell|\Big). 
\end{equation}
Let $f_j=e_j\setminus \cup_{\ell<j}e_\ell$.
By the frugality,  $|f_j|=|e_j\setminus\cup_{\ell<j}e_\ell|\ge 2$, so that
\begin{equation}\label{eq:ezboundej}
	|f_j|\le 2(|f_j|-1).
\end{equation}
Since for each~$j$, $e_j\in \ch_i$ for some $i\in L_j$, there are at most~$k-1$ many indices~$i'$ in~$[k]$ such that $i'\not\in L_j$
Then by~\eqref{eq:Vcontain},~\eqref{eq:sizeofspanMi}, and~\eqref{eq:ezboundej} we have
\begin{align*}
	|V|&\le t\sum_{i=1}^k\Big(\sum_{j:i\in L_j}(|f_j|-1)+\sum_{j:i\not\in L_j}|f_j|\Big)\\
	&\le t\sum_{i=1}^k\Big(\sum_{j:i\in L_j}(|f_j|-1)+\sum_{j:i\not\in L_j}2(|f_j|-1)\Big)\\
	&= t\sum_{j=1}^p\Big(\sum_{i:i\in L_j}(|f_j|-1)+\sum_{i:i\not\in L_j}2(|f_j|-1)\Big)\\
	&\le  t\sum_{j=1}^p \Big(  (|f_j|-1)+(k-1)2(|f_j|-1)         \Big)\\
	&= (2k-1)t\sum_{j=1}^p(|f_j|-1). 
\end{align*}
On the other hand, by~\eqref{eq:gammeEeq}
\[\eta(\C)=\eta(\I(\ch))\ge  \gamma^E(\ch)=\sum_{j=1}^p(|e_j\setminus\cup_{\ell<j}e_\ell|-1)  =\sum_{j=1}^p(|f_j|-1).\] 
We have thus shown
\[ \frac{|V(\C)|}{\eta(\C)}\le (2k-1)t. \]
Applying this argument to $\ch[S]$, for any $S \subseteq V$, yields

\[\cT(\C)=\max_{\emptyset\neq S\subseteq V}\frac{|S|}{\eta(\C[S])}\le (2k-1)t,\]
proving the theorem.
\end{proof}

\section{Fractional coloring and polytopes}\label{sec:fractionalcoloring}
We turn to ($\vw$-)fractional coloring (see \refD{def:fraccoloring}).
In this section and in its two successors we link ($\vw$-)fractional colorings to polytopes. 
The aim is to generalize the following result, proved in~\cite{matcomp},  to more than two matroids, and to vertex-weighted matroids. 
\begin{theorem}\label{thm:chistarmin}
For any pair of matroids $\M$ and $\N$ on the same ground set,
\begin{equation}\label{chistar}
	\chi^*(\M\cap \N)=\max(\chi^*(\M),\; \chi^*(\N)).\end{equation}
	\end{theorem}
	
	In fact, this was merely testimony to  ignorance --- the authors were unaware of \refT{thm:edmonds2matroidintersection}, of which it is a direct corollary. 
	

	\begin{observation}\label{entering_P}
For any complex $\C$, 
$\chi^*(\C) \le \frac{1}{t}$
if and only if 
$t\vOne \in P(\C)$.
\end{observation}

\begin{proof}
$t\vOne \in P(\C)$  means that $t\vOne 
\le \sum_{S\in\C} \lambda_S\mathbf{1}_{S}$ where $\lambda_S\ge 0$ and $\sum_{S\in\C} \lambda_S=1$, meaning that $\vOne \le \sum (\frac{\lambda_S}{t}) \mathbf{1}_{S}$ meaning that $\chi^*(\C) \le \sum \frac{\lambda_S}{t}= \frac{1}{t}$.
\end{proof}

Combined with \refT{thm:edmonds2matroidintersection}, this yields~\refT{thm:chistarmin}. There is nothing special about the vector~$\vec{1}$ --- the same argument works for every vector $\vw \in \mathbb{R}^V$. Below,~$\vw$ always denotes such a vector. It is also viewed as a weight function on~$V$, whence the arrow above is sometimes omitted.


For a closed down polytope $Z\subseteq\rvp$, let 
\[\psi(Z,\vh) := \min(\{t\in\mathbb{R}^+ \mid \vh/t \in Z\}).\]

\begin{theorem}\label{thm:psiformula}
$\chi^*(\C,\vh)= \psi(P(\C),\vh)$. 
\end{theorem}

\begin{proof}
Let $\psi(P(\C),\vh)=t$. Then $\frac{\vh}{t}\in P(\C)$ implies that there exists $\lambda_S\ge 0$ for all $S\in \C$ so that   $\vh=\sum_{S\in\C}t\lambda_S\mathbf{1}_{S}$, where  $\sum_{S\in\C}\lambda_S=1$. Then, the function $f:\C\rightarrow\mathbb{R}_{\ge 0}$ defined as $f(S):=t\lambda_S\ge 0$ satisfies that for each $v\in V(\C)$, $\sum_{S\in \C:v\in S}f(S)=\sum_{S\in\C}t\lambda_S\mathbf{1}_S(v)=w(v)$. Thus~$f$ is an $\vh$-fractional coloring.
Since $\sum_{S\in\C}f(S)=\sum_{S\in\C} t\lambda_S=t$, we have $\chi^*(\C,\vh)\le t=\psi(P(\C),\vh)$.

For the other direction, let $f:\C\rightarrow\mathbb{R}_{\ge 0}$ be an $\vh$-fractional coloring satisfying $\sum_{S\in\C}f(S)=\chi^*(\C,\vh)=a$. Let $\lambda_S=\frac{1}{a}f(S)$ for every $S\in\C$, and consider the function $g=\sum_{S\in\C}\lambda_S\mathbf{1}_S: V(\C)\rightarrow\mathbb{R}_{\ge 0}$. This is a convex combination of the $(\mathbf{1}_S)_{S\in\C}$ and thus $\vec{g}\in P(\C)$. Since~$f$ is an $\vh$-fractional coloring, $g(v)=\sum_{S\in\C:v\in S}\lambda_S=\sum_{S\in\C:v\in S}\frac{1}{a}f(S)\ge\frac{w(v)}{a}$ for each $v\in V(\C)$. So, $0\le \frac{\vh}{a}\le \vec{g}$. Since $P(\C)$ is closed-down, we have $\frac{\vh}{a}\in P(\C)$ and thus, $\psi(P(\C),\vh)\le a=\chi^*(\C,\vh)$.
\end{proof}

The theorem says that~$\chi^*(\C,\vh)$ is the smallest $t$ for which  $\vh/t \in P(\C)$.  On the other hand, for~$\EL=\{\M_1,\dots,\M_k\}\in \D^k$, to enter~$\cap_{i=1}^kP(\M_i)$ along the direction~$\vh$ towards the origin, you need to enter all ~$P(\M_i)$s. This is summarized in the following corollary, where as usual, $\lambda A=\{\lambda \vec{a} \mid \vec{a} \in A\}$.

\begin{corollary}\label{cor:Pchiequiv}
Let $\C_1,\dots, \C_k$ be complexes on~$V$ and $\C=\cap_{i=1}^k\C_i$. The following are equivalent for any $\lambda>0$:
\begin{enumerate}
	\item $\lambda P(\C)\supseteq\cap_{i=1}^kP(\C_i)$.
	\item $\chi^*(\C,\vh)\le \lambda\max_{1\le i\le k}\chi^*(\C_i,\vh)$ for every $\vh\in \rvp$.
\end{enumerate}
\end{corollary}




In the  spirit of \refC{cor:chiell2kchi} and \refCon{conj:chik-1chi}, we can ask for bounds on the ratio between the fractional chromatic number of the intersection of the  matroids and the fractional chromatic number of each matroid individually. 
It is  a viewpoint that gives rise to~\refCon{conj:chi*kint}. Another motivation of~\refCon{conj:chi*kint} is discussed in detail in~\refS{sec:weightedmatroid}.
\refCon{conj:chi*kint} is known for general~$k$  if~$k-1$ is replaced by~$k$ (see \refC{cor:kchi*}), and for $k=2$ as it is (this follows from  \refC{cor:Pchiequiv} and~\refT{thm:edmonds2matroidintersection}). In the next section we apply  a result from ~\cite{linhares2020approximate} to  show it for $k=3$. 



\section{Two more polytopes and polytope ratios}

This section discusses an equivalent version of~\refCon{conj:chi*kint}, involving polytopes.




Given a complex $\C$, let
\[Q(\C)=\{ f\in \mathbb{R}^{V(\C)}_{\ge 0} \mid f[S]\le rank_{\C}(S) \text{ for every }S\subseteq V(\C)\}.\]        

\begin{theorem}\cite{ed1970}\label{thm:qisp} 
	For a matroid $\M$ we have  $P(\M)= Q(\M).$
\end{theorem}

For general complexes the two polytopes can be arbitrarily far apart. For functions $f,g:V\rightarrow \mathbb{R}$, $f\cdot g$ is the inner product $\sum_{v\in V}f(v)g(v)$.
\begin{observation}\label{ob:lambdaPnotQ}
	For any ~$\lambda>0$, there exists a complex $\C$ such that  $\lambda P(\C)\not\supseteq Q(\C)$. 
\end{observation}
\begin{proof}
	We construct a complex~$\C$ and a vector $v\in\rvp$ such that $v\in Q(\C) \setminus \lambda P(\C)$.
	
	Let $k\in\mathbb{Z}^+$ be such that $\frac{1}{2}+\frac{1}{3}+\cdots+\frac{1}{k}> \lambda$, and let $n = 2+3+\cdots+k$. We construct the complex $\C$ on $V=[n]$ in the following way. 
	Let \[v = (\frac{1}{2}, \frac{1}{2}, \frac{1}{3}, \frac{1}{3}, \frac{1}{3}, \frac{1}{4}, \frac{1}{4}, \frac{1}{4}, \frac{1}{4}, ..., \frac{1}{k}, ..., \frac{1}{k}) \in \mathbb R^n.\]
	Let \[\C = \{ A\subseteq [n] :  v\cdot\mathbf{1}_A \le 1\}.\]
	
	Since any vector $u\in P(\C)$ is a convex combination of $\{\mathbf{1}_A\mid A\in \C\}$, then $v\cdot u\le 1$. Therefore $u\in \lambda P(\C)$ implies $v\cdot u\le \lambda$. Since
	\[v \cdot v = \sum_{i=2}^k \frac{1}{i}\cdot\frac{1}{i}\cdot i > \lambda,\] 
	we have $v \not\in \lambda P(\C)$.
	
	On the other hand, for $S \subseteq [n]$, let~$r$ the minimum integer such that $-1<v[S]\le r$.
	Then $S$ must have at least $r$ members with $v$-value at most $\frac{1}{r}$:
	if not, 
	\[v[S]=v[\{j\in S:v(j)>1/r\}]+v[\{j\in S: v(j)\le 1/r \}]   \le\sum_{i=2}^{r-1}\frac{1}{i}\cdot i+ r\frac{1}{r}\le r-1,\] a contradiction.
	Therefore the $r$ members form a face of~$\C$ and we have $rank_\C(S)\ge r\ge v[S]$,
	which implies $v \in Q(\C)$.     
\end{proof}

$P(\C)\neq Q(\C)$ is possible even  in flag complexes, as the following example shows. 
\begin{example}\label{ex:PnotQpartition}
	Define a complex~$\C$ on $V=\{x_1,x_2,\dots,x_9, y_1,y_2,y_3, z_1,z_2,z_3\}$ and the maximal faces of~$\C$ are $\{y_1,y_2,y_3\},\{ z_1,z_2,z_3 \}$, and $\{y_i, x_{3(i-1)+j}\}$, $\{z_j, x_{3(i-1)+j}\}$ for $1\le i,j\le 3$. It can be verified that this complex is 2-determined so it is a flag complex. Therefore by~\refProp{prop:four_equiv}, $\C$ is the intersection of partition matroids.
	
	Consider the vector $w$ such that $w(x_i)=\frac{1}{9}$ for each $1\le i\le 9$ and $w(y_j)=w(z_j)=\frac{1}{4}$ for each~$1\le j\le 9$.
	Then we claim that
	\[  w\in Q(\C)\setminus P(\C). \]
	
	To see that $w\in Q(\C)$, for every $U\subseteq V$ such that $rank_{\C}(U)=1$, the maximal such~$U$ in containment relation must be $\{x_t\}_{1\le t\le 9}$ or consist of four elements of $\{x_t\}_{1\le t\le 9}$, one elements of $\{y_1,y_2,y_3\}$, and one element of $\{z_1,z_2,z_3\}$. In either case we have $w[U]\le 1$. Since $rank(\C)=3$, for every set $U$ with $rank_\C(U)=3$, we have $w[U]\le w[V]\le \frac{1}{9}\cdot 9+\frac{1}{4}\cdot 6 \le 3$.
	The maximal rank 2 set in~$\C$ has the form 
	\[\{ y_{i}, y_{i'}, z_{j},z_{j'}, x_1,x_2,\dots, x_9  \}\] 
	for $i\neq i'$ and $j\neq j'$, which has $w$-weight $\frac{1}{9}\cdot 9+\frac{1}{4}\cdot 4 \le 2$, which implies for any~$U$ with $rank_{\C}(U)=2$, $w[U]\le 2$. Therefore $w\in Q(\C)$.
	
	To see that $w\not\in P(\C)$, we are going to prove that $w$ cannot be written as a convex combination $\sum \lambda_i\mathbf{1}_{S_i}$ for $S_i\in \C$. First, the coefficients of $\{y_1,y_2,y_3\}$ and $\{z_1,z_2,z_3\}$ cannot be positive, since the coefficients of the faces containing any $x_i$ for $1\le i\le 9$ should sum up to $\frac{1}{9}\cdot 9=1$. Second, for the sum of the coefficients of the faces of the form $\{x_p,y_i\}$ or $\{x_q,z_j\}$, when considering from the $\{x_t\}_{1\le t\le 9}$ side, the sum is at most $\frac{1}{9}\cdot 9=1$, but it should be at least $\frac{1}{4}\cdot 6=\frac{3}{2}$ when considering from the $\{y_1,y_2,y_3,z_1,z_2,z_3\}$ side, which is impossible. Therefore $w\not\in P(\C)$.
\end{example}

There is a third polytope associated with  $k$-tuples of matroids.
For $\EL=\{\M_1, \ldots ,\M_k\} \in \D^k$, let 
\[R(\EL)=\cap_{i=1}^k P(\M_i).\]

\refT{thm:qisp} implies  
\begin{equation}\label{eq:containmentPQR}
	P(\bigcap \EL)\subseteq Q(\bigcap \EL)\subseteq R(\EL). 
\end{equation}

\refT{thm:edmonds2matroidintersection} 
says that for $k=2$  equality holds throughout. 
By ~\refOb{ob:lambdaPnotQ} and ~\refE{ex:PnotQpartition}  (see also ~\refE{ex:truncatedPP} below) this is no longer true for $k>2$.

\begin{definition}\label{polytopesratio}
	Given  closed-down polytopes $A,B$ in $\rvp$, let 
	\[B:A=\min\{t\in\mathbb{R}_{\ge 0} \mid  tA \supseteq B\}\]
	(here, as usual, $\min \emptyset = \infty$).
\end{definition}

By~\refC{cor:Pchiequiv}, another formulation of~\refCon{conj:chi*kint} is the following.
\begin{conjecture}\label{conj:k-1PR}
	For any $\EL\in \D^k$,
	\[R(\EL):P(\bigcap\EL)\le k-1. \]
\end{conjecture}

\section{Polytope ratios and weighted matroids}\label{sec:weightedmatroid}

To study \refCon{conj:k-1PR}, we need vertex-weighted versions of matchings and cooperative coverings. For an algorithmic viewpoint of this subject see~\cite{frank}.
Throughout this section we have at hand a weight function~$w \in \mathbb{R}^V_{\ge 0}$.

Given a  matroid~$\M$ let~$\cnu_w(\M)$ be the maximum sum of weights in an edge of~$\M$. 
For a function  $f: V\rightarrow\mathbb{R}$ the \emph{$\cm$-span $f^\M$} of $f$ is a function, defined by
\[f^{\M}(v)=\max_{T\subseteq V: v\in span_{\M}(T)} \min_{u\in T}f(u).  \]

Otherwise put, 
\begin{equation*}
	f^\M(v) \ge t \Leftrightarrow v \in span_\M\{u \mid f(u) \ge t\}.
\end{equation*}


In particular, for $A\subseteq V$ we have   
$\mathbf{1}_A^\M= \mathbf{1}_{span_\M(A)}$.
A function~$f$ is said to be \emph{$w$-spanning} if $f^\M \ge w$. Let $\ctau^*_w(\M)=\min\{f[V]: f^\M \ge w\}$. A folklore result (see for example~\cite{matcomp}) is  $\ctau^*_w(\M)=\cnu_w(\M)$.

\begin{lemma}\label{lemma:independentvector}
	Let $\M$ be a matroid, $\vx \in P(\M)$ and $f \in \rvp$. Then $f^\M \cdot \vx \le f[V]$.
\end{lemma}

\begin{proof}
	By the definition of $P(\M)$    it suffices to prove the case that $x=\mathbf{1}_{S}$ for $S\in \M$. Assume $Im(f)=\{a_1,\dots, a_p\}$ for $0\le a_1<\cdots <a_p$.
	For~$\ell\in [p]$, let $Z_{\ell}=f^{-1}(a_\ell)$ and $U_\ell=span_\M(\cup_{\ell\le t\le p}Z_t)$.
	We prove the theorem by induction on~$p$. 
	
	When $p=1$, $f^\M\cdot x= a_1|S\cap U_1|$. Since $S\in \M$ and $U_1=span_\M(Z_1)$, 
	\[|S\cap U_1|\le  rank_\M(S\cap U_1)\le |Z_1|,\] 
	which implies that $f^\M\cdot x\le a_1 |Z_1|= f[V]$.
	
	For $p\ge 2$, define a function~$f_1$ such that for $v\in \cup_{\ell=1}^{p-1}Z_\ell$, $f_1(v)=f(v)$, and for $v\in Z_p$, $f_1(v)=a_{p-1}$.  Then
	\[f^{\M}\cdot x=f_1^\M\cdot x+ (a_p-a_{p-1})|S\cap U_p|. \]
	By induction hypothesis, $f_1^\M\cdot x\le f_1[V]$. On the other hand, $(a_p-a_{p-1})|S\cap U_p|\le (a_p-a_{p-1})rank_\M(U_p)\le (a_p-a_{p-1})|Z_p|$. Combining these implies that
	$f^\M\cdot x \le f_1[V]+(a_p-a_{p-1})|Z_p|=f[V]$.
\end{proof}
The concept of $w$-spanning has a cooperative version in ~$\D^k$. 
Given $\EL=\{\M_1,\dots,\M_k\}\in \D^k$, a $k$-tuple $(f_1, f_2, \dots ,f_k)$ of functions in $\rvp$ is a \emph{matroidal fractional $\vec{w}$-cooperative cover} if $\sum_{i=1}^kf_i^{\M_i}(v)\ge w(v)$ for every $v\in V$.
The \emph{matroidal fractional $\vec{w}$-covering number} $\ctau^*_w(\EL)$ is the minimum of $\sum_{i=1}^kf_i[V]$ over all matroidal fractional $\vec{w}$-cooperative covers $(f_1, f_2, \dots ,f_k)$.
The \emph{matroidal cooperative covering $\vec{w}$-number}~$\ctau_w(\EL)$ is the minimum of $\sum_{i=1}^kf_i[V]$ over all fractional $\vec{w}$-cooperative covers $(f_1, f_2, \dots f_k)$, where each~$f_i$ has integral values.
When $\vec{w}=\vec{1}$, we omit the subscript~$\vec{1}$. 
\begin{remark}\label{rmk:naming}
	The notational rule here and later is that a parameter~$z_w^*$ stands for the fractional $w$-weighted version, and~$z_w$ stands for the integral $w$-weighted version. We omit the mention of~$w$ when $\vec{w}=\vec{1}$.
\end{remark}

\begin{remark}
	This deviates from the notation in~\cite{matcomp}, where our~$\ctau_w^*$ is denoted by~$\tau_w$. Please also note the difference between this and fractional weighted covering and matching numbers in hypergraphs (see~\refD{def:fractionalmatching}),  where the  weights are on edges, rather than on vertices.  
\end{remark}


The \emph{matroidal fractional $w$-matching number}~$\cnu^*_w(\EL)$ is the maximum of~$\vx \cdot \vec{w}$ over all~$\vec{x}\in R(\EL)$.  
$\cnu_w(\EL)$,~$\cnu^*(\EL)$, and~$\cnu(\EL)$ are similarly defined.



When $w\equiv 1$ and the functions are integral we have $\cnu^*=\cnu$ and $\ctau^*=\ctau$, as defined in the introduction.

\begin{theorem}\label{thm:nu*w=tau*w} 
	For every $\EL \in \D^k$ we have   $\ctau^*_w(\EL)=\cnu^*_w(\EL)$.
\end{theorem}

\begin{proof}
	The inequality $\ctau^*_w(\EL)\ge\cnu^*_w(\EL)$ is a straightforward corollary of \refL{lemma:independentvector}: 
	Indeed, let $\vx \in R(\EL)=\cap_{i=1}^kP(\M_i)$ and let $(f_1, \ldots,f_k)$ be a matroidal fractional $w$-cooperative cover (so that $\sum_{i=1}^k f_i^{\M_i}\ge w$)   satisfying 
	$\sum_{i=1}^k f_i[V]=\ctau^*_w(\EL)$. Then by~\refL{lemma:independentvector},
	\[ w\cdot x \le \sum_{i=1}^k f^{\M_i}\cdot x \le \sum_{i=1}^k f_i[V]=\ctau^*_w(\EL).   \]
	As this is true for all $x\in R(\EL)$, it follows that $\cnu^*_w(\EL)\le \ctau^*_w(\EL)$.

	For the proof of the inverse inequality, note that by \refT{thm:qisp} the dual program of~$\cnu^*_w(\EL)$ is 
	\[\cnu_w^*(\EL)=\min \sum_{1\le i\le k, U\subseteq V}rank_{\M_i}(U)y_i(U)\] 
	subject to
	\begin{align*}
		\sum_{1\le i\le k, U\subseteq V:v\in U }y_i(U)\ge w(v) \text{ for every $v\in V$},\\
		y_i(U)\ge 0 \text{ for every $U\subseteq V$ and $1\le i\le k$.}
	\end{align*}
	Let $(y_1,\dots, y_k)$ be a feasible solution of the dual LP, minimizing
	\begin{equation}\label{eq:newcheck}
		\sum_{1\le i\le k, U\subseteq V} y_i(U) \cdot |U| |V\setminus U|. 
	\end{equation}
	Let $\mathcal{S}_i\subseteq 2^V$ be the support of $y_i$ for $1\le i\le k$. 
	We claim that for every $i\in [k]$, the collection $\mathcal{S}_i$ is a chain, i.e., for two distinct $S,T\in \mathcal{S}_i$, either $S\subseteq T$ or $T\subseteq S$.
	Suppose not, i.e., there exist $S,T\in\mathcal{S}_i$ such that~$S$ and~$T$ are incomparable in inclusion relation. Let $\alpha:=\min \Big( y_i(S),\; y_i(T)\Big)$ and define
	\begin{equation*}
		y^*_i(U)=\begin{cases}
			y_i(U)-\alpha & \text{ if $U\in \{S, T\}$,}\\
			y_i(U)+\alpha & \text{ if $U\in\{S\cap T, S\cup T \}  $,}\\
			y_i(U) &\text{ otherwise.}
		\end{cases}
	\end{equation*}
	Since
	\[ \mathbf{1}_{S}+\mathbf{1}_T=\mathbf{1}_{S\cap T}+\mathbf{1}_{S\cup T}, \]
	$(y_1,\dots, y_{i-1},y^*_i,y_{i+1},\dots, y_k)$ remains a feasible solution of the dual LP. And since
	\[ rank_{\M_i}(S)+rank_{\M_i}(T)\ge rank_{\M_i}(S\cap T) + rank_{\M_i}(S\cup T), \]
	it remains optimal. However it is easy to check that for the incomparable~$S$ and~$T$ (see, e.g.,~\cite[Theorem 2.1]{SchrijverCO}) 
	\[ |S||S^c|+|T||T^c|> |S\cap T| |(S\cap T)^c| +|S\cup T| |(S\cup T)^c|, \]
	where $U^c:=V\setminus U$, therefore the new optimal solution has a smaller sum~\eqref{eq:newcheck}, a contradiction. Therefore $\mathcal{S}_i$ is a chain.
	
	Assume $S_{i,1}\subseteq \cdots\subseteq  S_{i,\ell_i}$ is the chain in~$\mathcal{S}_i$. By the independence augmentation
	property of~$\M_i$, there exists a chain of independent sets $I_{i,1}\subseteq \cdots \subseteq I_{i,\ell_i}$ of~$\M_i$ such that $I_{i,t}\subseteq S_{i,t}$ and $rank_{\M_i}(S_{i,t})=|I_{i,t}|$ for each $t\in [\ell_i]$.
	We set 
	\[f_i(v):= \sum_{t=1}^{\ell_i}1_{\{v\in I_{i,t}\}}y_i(S_{i,t}), \]
	where $1_{\{v\in I_{i,t}\}}$ is the characteristic function of the event $\{v\in I_{i,t}\}$. Then
	\[f_i[V]= \sum_{t=1}^{\ell_i} rank_{\M_i}(S_{i,t})y_i(S_{i,t}) =\sum_{U\subseteq V} rank_{\M_i}(U)y_i(U), \]
	therefore 
	\[\sum_{i=1}^k f_i[V]=\cnu^*_w.\]
	It remains to prove that $(f_1,\dots, f_k)$ is a matroidal fractional $\vec{w}$-cooperative cover. Indeed, for any $v\in V$ and $i\in [k]$, in the chain $S_{i,1}\subseteq \cdots \subseteq S_{i,\ell_i}$ of~$\mathcal{S}_i$, assume~$\ell \in [\ell_i]$ is the minimum index such that $v\in S_{i,\ell}$. Then 
	\begin{equation}\label{eq:eachterm}
		\sum_{U\subseteq V: v\in U}y_i(U)=\sum_{\ell \le t\le \ell_i}y_i(S_{i,t}). 
	\end{equation}
	On the other hand, $v\in S_{i,\ell} \subseteq  span_{\M_i}(I_{i,\ell})$ and $\min_{u\in I_{i,\ell}}f(u)\ge \sum_{\ell\le t\le \ell_i}y_i(S_{i,t})$ by the choice of the chain~$(I_{i,t})_{t}$ and the definition of~$f_i$. Therefore
	\[ f^{\M_i}(v)\ge \sum_{\ell\le t\le \ell_i}y_i(S_{i,t}), \]
	which together with~\eqref{eq:eachterm} implies
	\[ \sum_{i=1}^k f^{\M_i}(v) \ge \sum_{1\le i\le k, U\subseteq V, v\in U}y_i(U)\ge w(v).  \]
	This completes the proof that $\ctau^*_w \le \cnu^*_w$, and thereby the theorem.
\end{proof}

\begin{corollary}\label{cor:nuwtotauw}
	$\cnu_w(\EL)\le \cnu^*_w(\EL)=\ctau^*_w(\EL)  \le \ctau_w(\EL)$.  
\end{corollary}

The gaps between these parameters are closely related to the ratios between the polytopes, as reflected in the following two results.


\begin{theorem}\label{thm:ratioRP=igap}
	$R(\EL):P(\bigcap\EL)=\max_{w\in\rvp}\ctau_w^*(\EL):\cnu_w(\EL)$.
\end{theorem}

\begin{proof}
	Let $R=R(\EL)$ and $P=P(\bigcap\EL)$. 
	Let $\lambda=\max_{w\in\rvp}\ctau_w^*(\EL):\cnu_w(\EL)$ and $\gamma = R(\EL):P(\bigcap\EL)$.
	
	Let us first show that $\gamma\le \lambda$.
	Assume for contradiction that $\gamma > \lambda.$ This means that there exists a vector $\vz\in R(\EL)\setminus \lambda P(\bigcap\EL)$. 
	We may clearly assume that each coordinate of~$\vz$ is strictly positive: suppose not, say $\vz(j)=0$.
	Since~$R$ is convex and $\{j\}\in\bigcap\EL$,
	then $\alpha \vz +(1-\alpha)\mathbf{1}_{\{j\}} \in R$ for every $\alpha\in [0,1]$. Since $\mathbf{1}_{\{j\}}\in \lambda P$ (as $\lambda \ge 1$) and~$\lambda P$ is closed, then there exists $\alpha_0\in (0,1)$ such that $\alpha_0 \vz +(1-\alpha_0)\mathbf{1}_{\{j\}} \in R\setminus\lambda P$.

	On the other hand, the convex polytope~$P$ is the intersection of a collection of half-spaces $\{ \vx \mid \vw_i\cdot \vx\le a_i \}$ for $i\in [m]$. Thus there exists some~$i$ such that $\vw_i\cdot \vz >a_i$.
	Since $\{j\}\in \bigcap\EL$ for every~$j\in V$, $\lambda P$ contains a standard simplex, so there exists some $t\in  (0,1)$ such that~$t\vz\in \lambda P$ and $\vw_i\cdot t\vz = a_i$. Then $\vw_i\ge 0$: otherwise suppose $\vw_i(j)<0$ for some~$j$, then $\vw_i\cdot (t\vz-se_j)> a_i$ for every $0<s<(t\vz)(j)$, contradicting the fact that~$\lambda P$ is closed down.

	The above argument implies $\ctau^*_{w_i}(\EL)=\cnu^*_{w_i}(\EL)> \lambda \cnu_{w_i}(\EL)$ for some $w_i\ge 0$, a contradiction to the choice of~$\lambda$.
	
	To show the inverse inequality,  
	for any~$w\in\rvp$, let $x\in R(\EL)$ be a matroidal fractional matching satisfying $w\cdot x=\cnu^*_w(\EL)$. Then by the assumption $\frac{x}{\gamma}\in P(\C)$ for $\C=\bigcap\EL$, which means there exists a collection of $\{S_i\in \C:i\in I\}$ such that $\frac{x}{\gamma}=\sum_{i\in I}\lambda_i \mathbf{1}_{S_i}$. We take $S_{t}$ in the collection such that $\mathbf{1}_{S_t}\cdot w =\max_{i\in I}\mathbf{1}_{S_i}\cdot w$. Therefore
	\[\frac{\cnu^*_w(\EL)}{\gamma}= \frac{x\cdot w}{\gamma}= \sum_{i\in I}\lambda_i\mathbf{1}_{S_i}\cdot w\le \mathbf{1}_{S_t}\cdot w\le \cnu_w(\EL), \]
	as desired.
\end{proof}

Next we compare $R(\EL)$ with $Q(\bigcap\EL)$.

\begin{theorem}\label{thm:ratioRQ=ratio}
	For every $k$ and~$\EL=\{\M_1,\dots,\M_k \}\in \D^k$, 
	\[R(\EL) : Q(\bigcap\EL)=\max_{U:U\subseteq V}\cnu^*(\EL_U):\cnu(\EL_U),\]
	where $\EL_U=\{\M_1[U],\dots,\M_k[U] \}$.
\end{theorem}
\begin{proof}
	Let $\EL=\{\M_1,\dots,\M_k\}$ and $\lambda =\max_{U:U\subseteq V}\cnu^*(\EL_U):\cnu(\EL_U)$. 
	\begin{claim}
		$R(\EL) : Q(\bigcap\EL)\le \lambda$.
	\end{claim}
	\begin{proof}[Proof of the claim]
		Let ~$x\in R(\EL)$.     We need to show that $x\in \lambda Q(\C)$, where $\C=\bigcap\EL$, namely 
		\begin{equation}\label{eq:desiredone}
			x[U]\le \lambda\cdot rank_\C(U)
		\end{equation}
		for  every $U\subseteq V$.

		Fix some $U\subseteq V$ and assume $\EL_U=\{\N_1,\dots,\N_k\}$, where $\N_i=\M_i[U]$. We define a vector~$\bar{x}\in\rvp$ as the following:
		$\bar{x}(v)=x(v)$ for $v\in U$ and $\bar{x}(v)=0$ for $v\in V\setminus U$. Then $x\in R(\EL)=\cap_{i=1}^kQ(\M_i)$ implies that for every $W\subseteq V$,
		\[ \bar{x}[W] = x[W\cap U]\le \min_{1\le i\le k} rank_{\M_i}(W\cap U)=\min_{1\le i\le k} rank_{\N_i}(W).  \]
		Thus $\bar{x}\in R(\EL_U)=\cap_{i=1}^kQ(\N_i)$. By the assumption of~$\lambda$, 
		$\cnu^*(\EL_U)\le \lambda \cnu(\EL_U)=\lambda \cdot rank_\C(U)$, and then
		\[ x[U]=\bar{x}[U]\le \cnu^*(\EL_U)\le \lambda \cnu(\EL_U)=\lambda\cdot rank_\C(U), \]
		as desired in~\eqref{eq:desiredone}.
	\end{proof} 
	Let $\gamma=R(\EL) : Q(\bigcap\EL)$. 
	\begin{claim}
		$\max_{U:U\subseteq V}\cnu^*(\EL_U):\cnu(\EL_U)\le \gamma$, i.e.,
		\begin{equation}\label{eq:desiredtwo}
			\cnu^*(\EL_U)\le\gamma\cnu(\EL_U)
		\end{equation}
		for every~$U\subseteq V$.  
	\end{claim}
	\begin{proof}[Proof of the claim]
		Fix $U\subseteq V$ and let $x\in R(\EL_U)$ satisfying $x[V]=\cnu^*(\EL_U)$. By the construction of~$\EL_U$, for every $W\subseteq V$,
		\[ x[W]=x[W\cap U]\le \min_{1\le i\le k}rank_{\M_i[U]}(W\cap U)\le \min_{1\le i\le k}rank_{\M_i}(W), \]
		which implies $x\in R(\EL)=\cap_{i=1}^kQ(\M_i)$. Then by the assumption,~$\frac{x}{\gamma}\in Q(\bigcap\EL)$.
		Therefore for every $W\subseteq V$,
		\[ \frac{x[W]}{\gamma}=\frac{x[W\cap U]}{\gamma}\le rank_{\bigcap\EL}(W\cap U)= rank_{\bigcap\EL_U}(W), \]
		which implies $\frac{x}{\gamma}\in Q(\bigcap\EL_U)$.
		Therefore
		\[ \frac{\cnu^*(\EL_U)}{\gamma}=  \frac{x[V]}{\gamma}\le rank_{\cap\EL_U}(V)=\cnu(\EL_U),  \]
		as desired in~\eqref{eq:desiredtwo}.  
	\end{proof}
	This completes the proof of the theorem.
\end{proof}


By~\refT{thm:ratioRQ=ratio},~\refCon{conj:matroidalryser} implies the following.
\begin{conjecture}\label{conjectureR:Q}
	If $\EL\in \D^k$, then
	\[   R(\EL):Q(\bigcap\EL)\le k-1.   \]
\end{conjecture}
At the risk of tiring the reader (no coercion to read) we mention two further generalizations, and  their fractional cooperative covers counterparts. 

\begin{conjecture}\label{conj:fksmatroid}
	For any $\EL\in \D^k$ 
	\[ \ctau_w^*(\EL)\le (k-1)\cnu_w(\EL).\]
\end{conjecture}
Or even stronger. 
\begin{conjecture}\label{conj:fksmatroid_integral} 
	If  $\vec{w}\in \mathbb{Z}^{V}_{\ge 0}$, then 
	\[ \ctau_w(\EL)\le (k-1)\cnu_w(\EL).\]
\end{conjecture}

By~Theorems~\ref{thm:ratioRP=igap} and~\ref{thm:nu*w=tau*w}, \refCon{conj:fksmatroid} is equivalent to~\refCon{conj:k-1PR}, and thus to~\refCon{conj:chi*kint}. Thus~\refCon{conj:fksmatroid} is stronger than \refCon{conjectureR:Q}, but~\refCon{conj:fksmatroid} is not comparable to~\refCon{conj:matroidalryser}. 
The ratio $\ctau_w^*(\EL)/ \cnu_w(\EL) =\cnu^*_w(\EL)/\cnu_w(\EL)$ is called the ``integrality gap'' of the linear program, so the conjecture is that the gap does not exceed $k-1$. \refCon{conj:fksmatroid} is true when $k\in\{2,3\}$ by~\refT{thm:edmonds2matroidintersection} and~\cite[Theorem 1]{linhares2020approximate}.
The proof for $k=2$ case (see e.g.,~\cite[Corolloary 41.12c]{SchrijverCO})  yields 
the validity of ~\refCon{conj:fksmatroid_integral} for this $k$.


We conclude this section with partial results on these conjectures.
\subsection{Weaker versions of  \refCon{conj:fksmatroid} and~\refCon{conj:fksmatroid_integral}}

\begin{theorem}\label{thm:tauw*knuw}
	For any $\EL\in \D^k$,   
	$ \ctau_w^*(\EL)\le k\cnu_w(\EL)$. When $w\in\mathbb{Z}^V_{\ge 0}$,  $ \ctau_w(\EL)\le k\cnu_w(\EL)$.
\end{theorem}
\begin{proof}
	Assume $\EL=\{\M_1,\cdots,\M_k\} \in \D^k$.  Let $I\in\M_1\cap\cdots\cap\M_k$ satisfy $w[I]=\cnu_w(\EL)$ and $w(u)>0$ for each $u\in I$. For each $1\le i\le k$, we set $f_i(x)=w(x)$ if $x\in I$ and $f_i(x)=0$ if $x\in V\setminus I$. Then $\sum_{i=1}^k f_i[V]=k\cnu_w$. It remains to verify that $\sum_{i=1}^k f_i^{\M_i}(v)\ge w(v)$ for each $v\in V$.
	Suppose not. Then there exists $v\in V$ such that $\sum_{i=1}^k f_i^{\M_i}(v)<w(v)$. Since for $x\in I$, $\sum_{i=1}^k f_i^{\M_i}(x)\ge \sum_{i=1}^k f_i(x)= kw(x)$, this~$v$ is not in $I$. Let $J$ be the set of indices $j$ such that $v\in span_{\M_j}(I)$. Then  $\sum_{i=1}^k f_i^{\M_i}(v) =\sum_{j\in J}f_j^{\M_j}(v)<w(v)$. It means for each $j\in J$ there exists circuit $C_j$ in $\M_j$ and $u_j\in C_j\setminus \{v \}$ such that $v\in C_j$ and $C_j\setminus\{v\}\subseteq I$, which satisfies
	\begin{equation}\label{eq:weightwuj}
		\sum_{j\in J}w(u_j)<w(v).
	\end{equation}
	For $j\in J$, 
	since $|C_j\setminus\{u_j\}|\le |I|$, by the independence augmentation axiom of~$\M_j$,~$C_j\setminus\{u_j\}$ can extend to $I\cup\{v\}\setminus\{u_j\}\in\M_j$. For $t\not\in J$, since $v\not\in span_{\M_t}(I)$, we have $I\cup\{v\}\in\M_t$. Therefore 
	\[I':=I\cup\{v\}\setminus \{u_j\mid j\in J\}\in \cap_{i=1}^k\M_i=(\cap_{j\in J}\M_j)\cap(\cap_{t\in [k]\setminus J}\M_t).\]
	But by~\eqref{eq:weightwuj}, $w(I')\ge w(I)-\sum_{j\in J}w(u_j)+w(v)>w(I)$,
	contradictory to the maximum weight assumption of~$I$, which completes the proof of the first part of the theorem.
	
	The same proof works when $w\in\mathbb{Z}^V_{\ge 0}$, in which case $(f_1,\dots, f_k)$ is an integral $w$-cooperative cover.
\end{proof}
By \refT{thm:ratioRP=igap} and~\refC{cor:Pchiequiv}, the theorem is equivalent to the following.
\begin{corollary}\label{cor:kRPratio}
	For any $\EL\in \D^k$,
	$R(\EL):P(\bigcap\EL) \le k$. 
\end{corollary}

\begin{corollary}\label{cor:kchi*}
	For any $\EL=\{\M_1, \cdots,\M_k\} \in \D^k$ and $\C=\bigcap\EL$, 
	\[\chi^*(\C,\vh)\le k\max_{1\le i\le k}\chi^*(\M_i,\vh).\] 
\end{corollary}

\begin{example}
	The following example shows the necessity of the requirement that $\M_1,\cdots,\M_k$ are matroids. Let $\C_1$ be the collection of rows and their subsets in the $n\times n$ grid, and ~$\C_2$  the collection of columns and their subsets.  Then $(P(\C_1)\cap P(\C_2)):P(\C_1\cap \C_2)= n$.
\end{example}


\subsection{The intersection of partition matroids }\label{subsec:partition}

\begin{definition}\label{def:fractionalmatching}
	Let $\ch$ be a hypergraph and $w: \ch\rightarrow \mathbb{R}_{\ge 0}$ be a function.
	\hfill
	\begin{enumerate} 
		\item    
		A \emph{fractional matching} in the hypergraph~$\ch$ is a function $g:\ch\rightarrow\mathbb{R}_{\ge 0}$ satisfying  \[\sum_{e:v\in e}g(e)\le 1\] for every $v\in V(\ch)$. 
		By $\nu^*_w(\ch)$ we denote the maximum value $\sum_{e\in\ch}w(e) g(e)$ over all fractional matchings~$g$. Furthermore~$\nu_w(\ch), \nu(\ch)$ are defined according to~\refR{rmk:naming}. 
		
		\item A {\em fractional $w$-cover} is a non-negative function~$t$
		on~$V(\ch)$ satisfying \[\sum_{v \in e}t(v)\ge w(e)\] for all $e \in \ch$. By $\tau^*_w(\ch)$ we denote the minimum~$g[V]$ over all fractional~$w$-covers~$g$. Furthermore~$\tau_w(\ch), \tau(\ch)$ are defined according to~\refR{rmk:naming}. 
	\end{enumerate}
\end{definition}

The following is a weighted generalization of Ryser's conjecture.

\begin{conjecture}[Weighted Ryser]\label{conj:weightedryser}
	Let $\ch$ be a $k$-partite hypergraph, and let $w$ be a non-negative integral function on $E(\ch)$. Then $\tau_w(\ch) \le (k-1)\nu_w(\ch)$.
\end{conjecture}
The conjecture was proved in~\cite{ryser3} in the case $k=3$ and $w \equiv 1$, using 
\refT{thm:tophall}. A fractional version of the weighted case was proved in \cite{furedikahnseymour}. 
So far, topological methods haven't been successfully applied to the weighted case.

F\"uredi proved a fractional version for general~$k$.

\begin{theorem}[\cite{furedi1981maximum}]\label{thm:furedi}
	In a $k$-partite hypergraph $\nu^*(\ch)\le (k-1)\nu(\ch)$.
\end{theorem}

As a  treat,  a short topological proof by Ori Kfir is given in~\refApp{app:Ori}.
F\"uredi, Kahn, and Seymour proved a weighted  version of F\"uredi's theorem: 

\begin{theorem}[\cite{furedikahnseymour}]\label{thm:FKS}
	If $\ch$ is $k$-partite then for every $w:\ch\rightarrow\mathbb{R}_{\ge 0}$, $\nu^*_w(\ch)\le (k-1) \nu_w(\ch)$.
\end{theorem}

\begin{observation}\label{ob:cnu=nupartition}
	If  $\EL=\{\M_1,\dots,\M_k\}\in\D^k$ is a $k$-set of partition matroids and $\ch=\K(\EL)$, then
	\[ \ctau^*_w(\EL)=\tau^*_w(\ch) \quad\text{and}\quad \cnu_w(\EL)=\nu_w(\ch).  \]
	And when $w\in\mathbb{Z}^V_{\ge 0}$, $\ctau_w(\EL)=\tau_w(\ch)$.
\end{observation}



Applying~\refT{thm:FKS} and~\refOb{ob:cnu=nupartition}, we prove~\refCon{conj:fksmatroid} and thus also~\refCon{conj:k-1PR} and~\refCon{conj:chi*kint}, for partition matroids.

\begin{corollary}
	If $\EL=\{\M_1,\dots,\M_k\}\in\D^k$  is a $k$-set of partition matroids, then 
	\[ \ctau_w^*(\EL)\le (k-1)\cnu_w(\EL).  \]
\end{corollary}






\begin{example}[Showing sharpness of~\refCon{conj:fksmatroid},~\refCon{conj:k-1PR}, and~\refCon{conj:chi*kint}]\label{ex:truncatedPP}
	The truncated projective plane~$T_k$ of uniformity~$k$ is the induced subhypergraph $P_k[V(P_k)\setminus\{v\}]$ for any vertex~$v$ of a projective plane~$P_k$ (assuming it exists).
	It is well-known that~$T_k$ is $k$-partite and each part is of size~$k-1$, each vertex is contained in~$k-1$ edges, $\nu(T_k)=1$, and $\nu^*(T_k)=\tau(T_k)=k-1$. Applying~\refOb{ob:cnu=nupartition} and~\refT{thm:nu*w=tau*w} proves the sharpness of~\refCon{conj:fksmatroid}, and hence of the other two conjectures.
	
\end{example}

\begin{remark} 
	By~\refProp{prop:every},
	\refR{ob:lambdaPnotQ} (or \refE{ex:PnotQpartition}) gives a complex~$\C$ such that $P(\C)\subsetneqq Q(\C)$ and $\C=\bigcap \EL_{1}$ for $\EL_1=\{\M_1,\dots, \M_{k_1}\}\in MINT_{k_1}$.
	\refE{ex:truncatedPP} gives a complex~$\D$ such that $Q(\D)\subsetneqq R(\EL_2)$ and $\D=\bigcap\EL_2$ for $\EL_2=\{\N_1,\dots,\N_{k_2}\}\in MINT_{k_2}$. Let $\B=\C*\D$. Then $\B= \bigcap\EL$ for a $(k_1+k_2)$-set~$\EL=\{\M_1',\dots, \M_{k_1}', \N_{1}',\dots, \N_{k_2}'  \}\in MINT_{k_1+k_2}$, where~$\M_i'=\M_i* 2^{V(\D)}$ and $\N_j'=2^{V(\C)}*\N_j$ for each $i\in [k_1]$ and $j\in [k_2]$.
	Then it is easy to see that $P(\B)\subsetneqq Q(\B) \subsetneqq R(\EL)$.
\end{remark}

\section{Fractional coloring and weighted topological expansion}\label{sec:weightedfractional}


The following was proved in~\cite{matcomp}.
\begin{theorem}\label{thm:chistardelta}
For a complex~$\C$,   $\chi^*(\C)\le \Delta(\C)$.
When~$\C$ is a matroid, equality holds. 
\end{theorem}

In this section,  we generalize this result  to vertex-weighted complexes.
For a complex~$\C$, let
\[\Delta(\C, \vh):=\max_{\emptyset\neq S\subseteq V(\C)}\frac{w[S]}{\min\Big(\eta(\C[S]),\; rank_\C(S)\Big)}.  \]
Note that $\Delta(\C)=\Delta(\C, \vec{1})$. 

\begin{remark}\label{rmk:MDelta}
For a loopless matroid $\M$, 
\[\Delta(\M, \vh)=\max_{\emptyset\neq S\subseteq V(\M)}\frac{w[S]}{rank_\M(S)}.  \]
\end{remark}

\refT{thm:chistardelta} is generalizable as follows. The proof involves carefully setting disjoint copies of the complexes and applying~\refT{thm:tophall}.

\begin{theorem}\label{thm:chi*com}
$\chi^*(\C,\vh)\le \Delta(\C,\vh)$. 
\end{theorem}
\begin{proof}
We have $\chi^*(\C,t\vh)=t\chi^*(\C,\vh)$ and $\Delta(\C,t\vh)=t\Delta(\C,\vh)$. Approximating $\vh$ by rational vectors we may therefore assume that~$\vh$ is integral, thus
$\Delta(\M,\vh)=\frac{p}{q}$ for integers $p\ge 0$ and $q>0$. Furthermore, we may assume that for every ~$v\in V$, $rank_\C(\{v\})=1$, since otherwise~$\Delta(\C,\vh)=\infty$ and the inequality holds.

Let~$W$ consist of ~$p$ disjoint copies of~$V(\C)$. We consider 
the join complex $\mathcal{D}=*_p\C$ on~$W$, and we consider a generalized partition matroid~$\N$ on~$W$ defined in the following way: for each~$v\in V(\C)$, the set~$W_v$ consisting of the~$p$ copies of~$v$ in~$W$ is a part, and the respective constraint parameter is~$w(v)q$.
Since $w(v)=\frac{w(v)}{rank_\C(\{v\})}\le \Delta(\C,\vh)=\frac{p}{q}$, we have $w(v)q\le p$. Thus~$\N$ is well-defined.

If there is a base~$S$ of~$\N$ that is in~$\D$, then~$S$ corresponds to~$p$ faces $S_1,\dots, S_p\in \C$ such that each $v\in V(\C)$ is in~$w(v)q$ of them. Then putting weight~$\frac{1}{q}$ on each of the faces~$S_1,\dots,S_p$ gives an $\vh$-fractional coloring of total weight $\frac{p}{q}=\Delta(\C,\vh)$, which completes the proof.

It remains to prove that there exists a base of~$\N$ belonging to ~$\D$.
Recall that the dual matroid $\N^*$ of $\N$ has the same ground set as~$\N$,  bases that are complements of bases  of~$\N$. Then~$\N^*$ is also a generalized partition matroid with parts~$(W_v)_{v\in V(\C)}$ and respective constraint parameters $(p-w(v)q)_{v\in V(\C)}$.

Let $U=\bigsqcup_2 W$ consist of two disjoint copies of~$W$, and consider the complex~$\D*\N^*$ on~$U$. For each $z\in W$, let~$U_z$ be the set of two copies of~$z$ in~$U$. If we can show that there exists a choice function $\phi: W\rightarrow\cup_{z\in W}U_z$ such that $\phi(z)\in U_z$ for each $z\in W$ and $Im(\phi)\in \D*\N^*$, then it corresponds to a set $A\subseteq W$ such that~$A\in\D$ and~$W\setminus A \in \N^*$. This implies that~$A$ contains a set~$S$ that is in~$\D$ and is a base in~$\N$. This proves the claim.

To show the existence of the choice function~$\phi$, we apply \refT{thm:tophall}.
The condition to be checked is that for every $X\subseteq W$,
\[  \eta\Big( (\D*\N^*)[\cup_{z\in X}U_z]\Big) \ge |X|. \]
By  \refProp{prop:etajoin}, 
\[\eta\Big( (\D*\N^*)[\cup_{z\in X}U_z]\Big)\ge  \eta(\D[X])+\eta(\N^*[X]) \]
so  it is enough to show that for every $X\subseteq W$,
\begin{equation}\label{eq:targetetabound}
	\eta(\D[X])+\eta(\N^*[X])\ge |X|.
\end{equation}

Fix some $X\subseteq W$.  For each~$v\in V(\C)$, let

\[ x(v)= |X\cap W_v| \]
be the number of copies of~$v$ occurring in~$X$. Clearly 
$0\le x(v)\le p$ and $\sum_{v\in V(\C)}x(v) =|X|$. 
For each $1\le i\le p$, let~$X_i$ be the intersection of~$X$ with the $i$th copy of~$V(\C)$. Therefore~$X$ is the disjoint union of~$X_i$'s.
By the definition of~$\Delta(\C,\vh)$, we have
\[  \frac{p}{q}= \Delta(\C,\vh) \ge \frac{w[X_i]}{\eta(\C[X_i])}    \]
so that $\eta(\C[X_i])\ge \frac{q}{p}w[X_i]$.
Again by \refProp{prop:etajoin},
\begin{equation}\label{eq:etaDX}
	\eta(\D[X])\ge \sum_{i=1}^p \eta(\C[X_i])\ge \sum_{i=1}^p\frac{q}{p}w[X_i]=\sum_{v\in V(\C)}\frac{q}{p}x(v) \cdot w(v). 
\end{equation}
On the other hand, by~\refProp{prop:etamatroid} 
\begin{equation}\label{eq:etaN*X}
	\begin{split}
		\eta(\N^*[X]) \ge rank_{\N^*}(X) &=\sum_{v\in V(\C)}\min\Big(p-w(v)q,\; x(v)\Big)\\
		&\ge \sum_{v\in V(\C)}(p-w(v)q)\frac{x(v)}{p}.
	\end{split}
\end{equation}
Combining~\eqref{eq:etaDX} and~\eqref{eq:etaN*X}, we have
\[ \eta(\D[X]) + \eta(\N^*[X]) \ge\sum_{v\in V(\C)}x(v) =|X|,   \]
which proves~\eqref{eq:targetetabound}.
\end{proof}

\begin{theorem}\label{thm:chi*Ch}
For any complex~$\C$,
\[\chi^*(\C,\vh)\ge\Delta_r(\C,\vh):=\max_{\emptyset\neq S\subseteq V}\frac{w[S]}{rank_\C(S)}.\]
\end{theorem}
\begin{proof}
Let $X\subseteq V$ satisfy $\frac{w[X]}{rank_\C(X)}=\Delta_r(\C,\vh)$. Let $f$ be any $\vh$-fractional coloring of~$\C$. Then
\[ \sum_{v\in X}\sum_{S\in \M:v\in S}f(S)\ge w[X].  \]
On the other hand, 
\[ \sum_{v\in X}\sum_{S\in \C:v\in S}f(S)\le \sum_{S\in \C} \sum_{v\in X\cap S}f(S)\le rank_\C(X)\sum_{S\in \C}f(S),  \]
which implies $\sum_{S\in\C}f(S)\ge \frac{w[X]}{rank_\C(X)}=\Delta_r(\C,\vh)$.
\end{proof}
Combining~\refT{thm:chi*com}, \refR{rmk:MDelta}, and~\refT{thm:chi*Ch}, we have the following result.
\begin{corollary}
For any matroid $\M$,
$\chi^*(\M,\vh)=\Delta(\M,\vh)$.
\end{corollary}


\section{Fractional list colorings of complexes}\label{sec:fractionallist}

We conclude with a combination of the  two themes of the paper, list colorings and fractional colorings. In \cite{ALON199731}  the notion of fractional list coloring of graphs was introduced. 
Here we generalize it  to complexes.
Given a complex $\C$, we say that $\C$ is \emph{$(a,b)$-colorable} for some inegers $a\ge b\ge 1$, if for $L_v=\{1,\dots, a\}$ for each $v\in V(\C)$, there are subsets $C_v\subseteq L_v$ with $|C_v|=b$ for each $v\in V(\C)$ such that $S_i:=\{v\in V(\C):i\in C_v\}$ is in $\C$ for each $i\in \{1,\dots, a\}$. Let $CL(\C):=\{(a,b):\C \text{ is $(a,b)$-colorable}\}$ and let the \emph{colorable ratio} be $clr(\C):=\inf\{\frac{a}{b}:(a,b)\in CL(\C)\}$.

Clearly, $(\chi(\C),1)\in CL(\C)$ so that $clr(\C)\le \chi(\C)$.

Given a complex $\C$, we say that $\C$ is \emph{$(a,b)$-choosable} for some inegers $a\ge b\ge 1$, if for any size~$a$ lists $(L_v:v\in V(\C))$, there are subsets $C_v\subseteq L_v$ with $|C_v|=b$ for each $v\in V(\C)$ such that $S_i:=\{v\in V(\C):i\in C_v\}$ is in $\C$ for each $i\in \cup_{v\in V(\C)}L_v$. Let $CH(\C):=\{(a,b):\C \text{ is $(a,b)$-choosable}\}$ and let the \emph{choice ratio} be $chr(\C):=\inf\{\frac{a}{b}:(a,b)\in CH(\C)\}$.

Clearly, $(\chi_\ell(\C),1)\in CH(\C)$, so  $chr(\C)\le \chi_\ell(\C)$.

\begin{theorem}\label{thm:fracchrclr}
For any complex~$\C$,
$\chi^*(\C)=clr(\C)=chr(\C)$.
\end{theorem}

The proof is similar to that in ~\cite{ALON199731}. Details are given in \refApp{app:choose}.

\bigskip{\noindent\bf Acknowledgements.} 
We would like to thank the reviewers for their astute comments.

\small
\bibliographystyle{abbrv}

\normalsize

\appendix

\section{A topological proof of~\refT{thm:furedi}}\label{app:Ori}

As a treat, we give a nice proof of~\refT{thm:furedi},  by  Ori Kfir in his M.Sc thesis ~\cite{orithesis}. 

\begin{remark}
The proof yields the result also for ``bipartite" hypergraphs, namely those having a set that meets every edge at one vertex. This is not an innovation, since F\"uredi's original theorem contains also this case, because the condition it needs is non-containment of a projective plane, which obviously holds for ``bipartite" hypergraphs.  
\end{remark}

A standard  argument of adding    dummy elements yields a  deficiency  version of \refT{thm:tophall}.

\begin{theorem}\label{thm:deficiency}
Let $V_1, \ldots, V_m$ sets of vertices of a complex $\C$. If $\eta(\C[\cup_{i\in I}V_i]) \ge |I|-d$ for every $I \subseteq [m]$ then there exists a partial transversal of size $m-d$, i.e., a choice function $\phi$ from $S\subseteq [m]$ of size $(m-d)$ to $\bigcup_{i\in S}F_i$ , such that $\phi(i)\in F_i$ for all $i\in S$ and $Im( \phi) \in \C$.
\end{theorem}

\begin{theorem}[Theorem 7.8 in~\cite{orithesis}]
Let $\ch$ be a $k$-uniform hypergraph. Suppose the vertex set of~$\ch$ can be divided into~$U$ and~$U'$ such that for every edge $S\in\ch$, $|S\cap U|=1$. Then
\[ \nu^*(\ch)\le (k-1)\nu(\ch). \]
\end{theorem}
\begin{proof}
Let
$\ch'=\ch/U=\{S \setminus U \mid S \in\ch \}$, which is a $(k-1)$-uniform hypergraph.
For every $u \in U$ let $V_u=\{S \in \ch' \mid S \cup \{u\} \in \ch\}$. Let 
\[d =  \max_{W \subseteq U}\Big( |W| -\eta (\M[\cup_{w\in W}V_w])\Big),\]
and let $Z$ be the subset of $U$ at which the maximum is attained. Putting $W=\emptyset$ shows $d\ge 0$. By~\refT{thm:deficiency}, $\nu(\ch) \ge |U|-d$. On the other hand, by our choice, $\eta(\M[\cup_{w\in Z}V_w])= |Z|-d$. By \refT{thm:abm} $\eta(\M[\cup_{w\in Z}V_w])\ge \tau^*(\cup_{w\in Z}V_w).$ Hence 
\[(k-1)\eta(\M[\cup_{w\in Z}V_w])=  (k-1)(|Z|-d)\ge\tau^*(\cup_{w\in Z}V_w). \]
Let $f$ be a fractional cover of $\cup_{w\in Z}V_w$ of total weight at most $(k-1)(|Z|-d)$. Then  $f+\mathbf{1}_{U\setminus Z}$ is  a fractional cover of~$\ch$, which has total weight at most 
\[(k-1)(|Z|-d)+|U|-|Z|\le (k-1)(|Z|-d+|U|-|Z|)= (k-1)(|U|-d)=(k-1)\nu(\ch). \]
Therefore
$\nu^*(\ch)=\tau^*(\ch)\le (k-1)\nu(\ch)$.
\end{proof}

\section{Proof of~\refT{thm:fracchrclr}}\label{app:choose}

\begin{lemma}
For any complex $\C$, $clr(\C)\ge \chi^*(\C)$.
\end{lemma}
The proof is essentially same as the proof of lower bound in~\cite[Section 4]{ALON199731}.
\begin{lemma}
For a complex $\C$, $chr(\C)\ge clr(\C)$.
\end{lemma}
We omit the straightforward proof.

\begin{lemma}
For a complex $\C$, $\chi^*(\C)\ge chr(\C)$.
\end{lemma}
\begin{proof}
The first observation is that
the optimum of the linear program of fractional coloring can be obtained by some basic feasible solution~$f$, and since the constraints of fractional coloring are integral, by Cramer's rule, there exists 
a subset~$\C_0$ of~$\C$ and a positive integer~$q$ such that 
for each $T\in\C_0$, $f(T)=p_T/q$ for some positive integer~$p_T$, and for each $T\in\C\setminus\C_0$, $f(T)=0$. And
\[\chi^*(\C)=f[\C]=f[\C_0]=\frac{p}{q},\] where $p:=\sum_{T\in\C_0}p_T$. By taking each $T\in\C_0$ $p_T$ times, we obtain a collection of faces $\{T_1,\dots, T_p\}$ (not necessarily distinct) such that each vertex $v\in V(\C)$ is in at least $q$ many of them: 
\[|\{j:v\in T_j\}|=\sum_{T\in\C_0:v\in T}p_T=q\sum_{T\in\C:v\in T}f(T)\ge q,  \]
where the last inequality is by the constraint of the linear program.

Then to prove $\frac{p}{q}=\chi^*(\C)\ge chr(\C)$, we have to show that for every $\epsilon>0$ there exists $(a,b)\in CH(\C)$ such that $\frac{a}{b}\le (1+\epsilon)\frac{p}{q}$. Consider $a:=(1+\epsilon)pt$ and $b:=qt$, for $t$ sufficiently large, and we can assume without loss of generality that $a,b$ are integers.
For any size~$a$ lists $(L_v:v\in V(\C))$, we take a random partition $\cup_{v\in V(\C)}L_v=Z_1\cup\dots Z_p$ such that each color is included in each~$Z_j$ independently with probability~$1/p$.
Noting that~$p$ and~$|V(\C)|$ are fixed for any given~$\C$, it is followed by the well-known Chernoff's bounds that with positive probability, for each $v\in V(\C)$ and $j\in\{1,\dots,p\}$,
$|L_v\cap Z_j|$ is concentrated around its expectation $|L_v|/p=(1+\epsilon)t$. Thus we can take an $t$-subset $C_{v,j}$ of $L_v\cap Z_j$. Then we get a set
\[ C_v:=\cup_{1\le j\le p:v\in T_j}C_{v,j} \]
of size at least $qt=b$. 

Then for each color $i\in\cup_{v\in V(\C)}L_v$ and~$S_i=\{v\in V(\C)\mid i\in C_v\}$, if $i\not\in\cup_{v\in V(\C)}C_v$, then $S_i=\emptyset$, otherwise~$S_i\subseteq T_j$ for the index~$j$ satisfying~$i\in Z_j$. Since~$T_j\in\C$, in either case~$S_i\in \C$ and we prove~$(a,b)\in CH(\C)$.
\end{proof}

Combining the above lemmas, we have
\[\chi^*(\C)= chr(\C)=clr(\C), \]
which completes the proof of~\refT{thm:fracchrclr}.

\end{document}